\documentclass[12pt]{amsart}

\usepackage{amsmath,amsthm,amssymb}
\usepackage[
margin=1in,
includefoot,
footskip=30pt,
]{geometry}
\usepackage{layout}
\usepackage{graphicx}
\usepackage[capitalise,nameinlink,noabbrev]{cleveref}
\hypersetup{final}
\usepackage{epstopdf}
\usepackage{subcaption}
\usepackage{stmaryrd}
\usepackage{tikz}
\usetikzlibrary{patterns}

\newtheorem{lemma}{Lemma}

\newtheorem{theorem}{Theorem}

\newtheorem{example}{Example}
\newtheorem{corollary}{Corollary}

\def\R{\mathbb R}
\def\N{\mathbb N}
\def\P{\mathbb P}
\def\J{\mathcal J}

\def\p{\partial}
\DeclareMathOperator{\supp}{supp}
\newcommand{\pair}[1]{\left\langle #1 \right\rangle}
\newcommand{\norm}[1]{\left\|#1 \right\|}

\newcommand{\tnorm}[1]{{\left\vert\kern-0.25ex\left\vert\kern-0.25ex\left\vert #1 
		\right\vert\kern-0.25ex\right\vert\kern-0.25ex\right\vert}}

\let\div\relax
\DeclareMathOperator{\div}{div}
\def\grad{\nabla}

\date{\today}

\title[Unique continuation for the Helmholtz equation using stabilized FEM]{Unique continuation for the Helmholtz equation using stabilized finite element methods}

\usepackage[foot]{amsaddr}

\author{Erik Burman}
\author{Mihai Nechita}
\author{Lauri Oksanen}

\address{Department of Mathematics, University College London, Gower Street, London UK, WC1E 6BT.}
\email{\{e.burman, mihai.nechita.16, l.oksanen\}@ucl.ac.uk}

\begin{document}

\begin{abstract}
	In this work we consider the computational approximation of a unique continuation problem for the Helmholtz equation using a stabilized finite element method. First conditional stability estimates are derived for which, under a convexity assumption on the geometry, the constants grow at most linearly in the wave number. Then these estimates are used to obtain error bounds for the finite element method that are explicit with respect to the wave number. Some numerical illustrations are given.
\end{abstract}	
	
\maketitle

\section{Introduction}

We consider a unique continuation (or data assimilation) problem for the Helmholtz equation
\begin{align}\label{helholtz_intro}
\Delta u + k^2 u = -f,
\end{align}
and introduce a stabilized finite element method (FEM) to solve the problem computationally. 
Such methods have been previously studied for Poisson's equation in \cite{BurmanSIAM}, \cite{BurmanComptes} and \cite{BHL2018}, and for the heat equation in \cite{BurmanOksanen}. 
The main novelty of the present paper is that our method is robust with respect to the wave number $k$, and we prove convergence estimates with explicit dependence on $k$, see \cref{L2error} and \cref{H1error} below. 

An abstract form of a unique continuation problem is as follows. 
Let $\omega \subset B \subset \Omega$ be open, connected and non-empty sets in $\R^{1+n}$ and suppose that $u \in H^2(\Omega)$ satisfies (\ref{helholtz_intro}) in $\Omega$.
Given $u$ in $\omega$ and $f$ in $\Omega$, find $u$ in $B$.

This problem is non-trivial since no information on the boundary $\p \Omega$ is given. 
It is well known, see e.g. \cite{IsakovBook},
that if $\overline{B \setminus \omega} \subset
\Omega$ then the problem is conditionally H\"older stable: for all $k \ge 0$ there are $C > 0$ and $\alpha \in (0,1)$
such that for all $u \in H^2(\Omega)$
\begin{align}\label{stability_intro}
\norm{u}_{H^1(B)} 
\le 
C (\norm{u}_{H^1(\omega)} + 
\norm{\Delta u + k^2 u}_{L^2(\Omega)})^\alpha \norm{u}_{H^1(\Omega)}^{1-\alpha}.
\end{align}
If $B \setminus \omega$ touches the boundary of $\Omega$, then one can only expect logarithmic stability, since it was shown in the classical paper \cite{John} that the optimal stability estimate for analytic continuation from a disk of radius strictly less than 1 to the concentric unit disk is of logarithmic type, and analytic functions are harmonic.

In general, the constants $C$ and $\alpha$ in \eqref{stability_intro} depend on $k$, as can be seen in \cref{wkb} given in \ref{appendix}. However, under suitable convexity assumptions on the geometry and direction of continuation it is possible to prove that in (\ref{stability_intro}) both the constants $C$ and $\alpha$ are independent of $k$, see the uniform estimate in \cref{cor_Holder} below, which is closely related to the so-called increased stability for unique continuation
\cite{Isakov}. Obtaining optimal error bounds in the finite element approximation crucially depends on deriving estimates similar to (\ref{stability_intro}), with weaker norms in the right-hand side, as in \cref{cor_Holder_impr} below, or in both sides, by shifting the Sobolev indices one degree down, as in \cref{shifted3b} below.

In addition to robustness with respect to $k$,
an advantage of using stabilized FEM for this unique continuation problem is that--when designed carefully--its implementation does not require information on the constants $C$ and $\alpha$ in (\ref{stability_intro}), or any other quantity from the continuous stability theory, such as a specific choice of a Carleman weight function.
Moreover, unlike other techniques such as Tikhonov regularization or quasi-reversibility, no auxiliary regularization parameters need to be introduced.
The only asymptotic parameter in our method is the size of the finite element mesh, and in particular, we do not need to saturate the finite element method with respect to an auxiliary parameter as, for example, in the estimate (34) in \cite{Bourgeois}.

Throughout the paper, $C$ will denote a positive constant independent of the wave number $k$ and the mesh size $h$, and which depends only on the geometry of the problem. By $A \lesssim B$ we denote the inequality $A \le C B$, where $C$ is as above.

For the well-posed problem of the Helmholtz equation with the Robin boundary condition
\begin{equation}\label{helmholtzbvp}
\Delta u + k^2 u = -f \quad \text{in } \Omega \quad \text{and} \quad \p_n u + \mathrm{i} ku = 0 \quad \text{on } \p\Omega, 
\end{equation}
the following sharp bounds
\begin{equation}\label{well-posed1}
\norm{\nabla u}_{L^2(\Omega)} + k\norm{u}_{L^2(\Omega)} \le C \norm{f}_{L^2(\Omega)}
\end{equation}
and
\begin{equation}\label{well-posed2}
\norm{u}_{H^2(\Omega)} \le C k \norm{f}_{L^2(\Omega)}
\end{equation}
hold for a star-shaped Lipschitz domain $\Omega$ and any wave number $k$ bounded away from zero \cite{Spence2016}. The error estimates that we derive in \cref{stabilizedfem}, e.g. $\norm{u-u_h}_{H^1(B)} \le C (h k)^\alpha \norm{u}_*$ in \cref{H1error}, contain the term 
\begin{equation}\label{starnorm}
\norm{u}_* = \norm{u}_{H^2(\Omega)} + k^2 \norm{u}_{L^2(\Omega)},
\end{equation}
which corresponds to the well-posed case term $k \norm{f}_{L^2(\Omega)}$.

It is well known from the seminal works \cite{BabuskaRev, IhlenburgBabuska1, IhlenburgBabuska2} that the finite element approximation of the Helmholtz problem is challenging also in the well-posed case due to the so-called pollution error. Indeed, to observe optimal convergence orders of $H^1$- and $L^2$-errors the mesh size $h$ must satisfy a smallness condition related to the wave number $k$, typically for piecewise affine elements, the condition $k^2 h \lesssim 1$.
This is due to the dispersion error that is most important for low order approximation spaces. The situation improves if higher order polynomial approximation is used. Recently, the precise conditions for optimal convergence when using $hp$-refinement ($p$ denotes the polynomial order of the approximation space) were shown in \cite{Melenk}.
Under the assumption that the solution operator for Helmholtz problems is polynomially bounded in $k$, it is shown that quasi-optimality is obtained under the conditions that $kh/p$ is sufficiently small and the polynomial degree $p$ is at least $O(\log k)$.

Another way to obtain absolute stability (i.e. stability without, or under mild, conditions on the mesh size) of the approximate scheme is to use stabilization. The continuous interior penalty stabilization (CIP) was introduced for the Helmholtz problem in \cite{Wu}, where stability was shown in the $k h \lesssim 1$ regime, and was subsequently used to obtain error bounds for standard piecewise affine elements when $k^3 h^2 \lesssim 1$. It was then shown in \cite{BurmanHelmholtz} that, in the one dimensional case, the CIP stabilization can also be used to eliminate the pollution error, provided the penalty parameter is appropriately chosen. When deriving error estimates for the stabilized FEM that we herein introduce, we shall make use of the mild condition $k h \lesssim 1$. To keep down the technical detail we restrict the analysis to the case of piecewise affine finite element spaces, but the extension of the proposed method to the high order case follows using the stabilization operators suggested in \cite{BurmanSIAM} (see also \cite{BurmanChapter} for a discussion of the analysis in the ill-posed case). 

From the point of view of applications, unique continuation problems often arise in control theory and inverse scattering problems. For instance, the above problem could arise when the acoustic wave field $u$ is measured on $\omega$
and there are unknown scatterers present outside $\Omega$.

\section{Continuum stability estimates}\label{continuum_estimates}

Our stabilized FEM will build on certain variations of the basic estimate (\ref{stability_intro}),
with the constants independent of the wave number,
and we derive these estimates in the present section. 
The proofs are based on a Carleman estimate that is a variation of \cite[Lemma 2.2]{Isakov}
but we give a self-contained proof for the convenience of the reader. 
In \cite{Isakov} the Carleman estimate was used to derive a so-called increased stability estimate under suitable convexity assumptions on the geometry. To be more precise, let $\Gamma \subset \partial \Omega$ be such that $\Gamma \subset \partial \omega$ and $\Gamma$ is at some positive distance away from $\partial \omega \cap \Omega$. For a compact subset $S$ of the open set $\Omega$, let $P(\nu; d)$ denote the half space which has distance $d$ from $S$ and $\nu$ as the exterior normal vector. Let $\Omega(\nu;d) = P(\nu; d) \cap \Omega$ and denote by $B$ the union of the sets $\Omega(\nu;d)$ over all $\nu$ for which $P(\nu; d) \cap \partial \Omega \subset \Gamma$. This geometric setting is exemplified by \cref{domain_conv} and it is illustrated in a general way in Figures 1 and 2 of \cite{Isakov} where $B$ is denoted by $\Omega(\Gamma;d)$. Under these assumptions it was proven that
\begin{align}
\label{increased_stability}
\norm{u}_{L^2(B)} 
\le 
C F + C k^{-1} F^\alpha \norm{u}_{H^1(\Omega)}^{1-\alpha},
\end{align}
where $F = \norm{u}_{H^1(\omega)} + 
\norm{\Delta u + k^2 u}_{L^2(\Omega)}$
and the constants $C$ and $\alpha$ are independent of $k$.
Here $F$ can be interpreted as the size of the data in the unique continuation problem and the $H^1$-norm of $u$
as an a priori bound. 
As $k$ grows, the first term on the right-hand side of (\ref{increased_stability}) dominates the second one, 
and the stability is increasing in this sense. 

As our focus is on designing a finite element method,
we prefer to measure the size of the data in the weaker norm 
$$
E = \norm{u}_{L^2(\omega)} + 
\norm{\Delta u + k^2 u}_{H^{-1}(\Omega)}.
$$
Taking $u$ to be a plane wave solution to (\ref{helholtz_intro}) suggests that 
\begin{align*}
\norm{u}_{L^2(B)} 
\le 
C k E + C E^\alpha \norm{u}_{L^2(\Omega)}^{1-\alpha},
\end{align*}
could be the right analogue of (\ref{increased_stability})
when both the data and the a priori bound are in weaker norms. 
We show below, see \cref{shifted3b}, 
a stronger estimate with only the second term on the right-hand side.

\cref{lem_carleman_eq} below captures the main step of the proof of our Carleman estimate. This is an elementary, but somewhat tedious, computation that establishes an identity similar to that in \cite{Liu} where the constant in a Carleman estimate for the wave equation was studied. 
For an overview of Carleman estimates see \cite{LeRousseau,TataruRev}, the classical references are \cite[Chapter 17]{HormanderVol3} for second order elliptic equations, and \cite[Chapter 28]{HormanderVol4} for hyperbolic and more general equations.
In the proofs, the idea is to use an exponential weight function $e^{\ell(x)}$ and study the expression
$$
\Delta (e^\ell w) = e^\ell \Delta w + \text{lower order terms},
$$
or the conjugated operator $e^{-\ell} \Delta e^\ell$.
A typical approach is to study commutator estimates for the real and imaginary part of the principal symbol of the conjugated operator,
see e.g. \cite{LeRousseau}. This can be seen as an alternative way to estimate the cross terms appearing in the proof of \cref{lem_carleman_eq}. 
Sometimes semiclassical analysis is used to derive the estimates, see e.g. \cite{LeRousseau}. This is very convenient when the estimates are shifted in the Sobolev scale, and we will use these techniques in \cref{sec_shifting} below.

\subsection{A Carleman estimate and conditional H\"older stability}

Denote by $(\cdot, \cdot)$, $|\cdot|$, $\div$, $\nabla$ and $D^2$ the inner product, norm,  divergence, gradient and Hessian with respect to the Euclidean structure in $\Omega \subset \R^{1+n}$.
(Below, \cref{lem_carleman_eq} and \cref{cor_ptwise_Carleman} are written so that 
they hold also when $\Omega$ is a Riemannian manifold
and the above concepts are replaced with their Riemannian analogues.)

\begin{lemma}
	\label{lem_carleman_eq}
	Let $k \ge 0$.
	Let $\ell, w \in C^2(\Omega)$ and $\sigma \in C^1(\Omega)$.
	We define $v = e^\ell w$, and 
	\begin{align*}
	a = \sigma - \Delta \ell, \quad 
	q = k^2 + a + |\nabla \ell|^2, \quad 
	b = -\sigma v - 2(\grad v, \grad \ell), \quad
	c = (|\grad v|^2 - q v^2) \grad \ell.
	\end{align*}
	Then
	\begin{align*}
	&e^{2 \ell} (\Delta w + k^2 w)^2/2 
	=
	(\Delta v + q v)^2/2 + b^2/2
	\\&\quad
	+ a |\nabla v|^2 + 2 D^2 \ell(\grad v, \grad v)
	+ \left(-a |\grad \ell|^2 + 2 D^2 \ell (\grad \ell, \grad \ell)\right)v^2
	- k^2 a v^2
	\\&\quad
	+ \div(b \nabla v + c) + R,
	\end{align*}
	where 
	$
	R = (\grad \sigma , \grad v)v + \left(\div (a \grad \ell) - a\sigma \right) v^2.
	$
\end{lemma}

A proof of this result is given in \ref{appendix}. In the present paper we use \cref{lem_carleman_eq} only with the choice $\sigma =
\Delta \ell$, or equivalently $a=0$, but the more general version of the lemma is useful when non-convex geometries are considered. In fact, instead of using a strictly convex function $\phi$ as in \cref{cor_ptwise_Carleman} below, it is possible to use a function $\phi$ without critical points, and convexify by taking $\ell = \tau e^{\alpha \phi}$ and $\sigma =
\Delta \ell + \alpha \lambda \ell$ for suitable constants $\alpha$ and $\lambda$. In the present context this will lead to an estimate that is not robust with respect to $k$, but we will use such a technique in the forthcoming paper \cite{BNO2018}.

\begin{corollary}[Pointwise Carleman estimate]
	\label{cor_ptwise_Carleman}
	Let $\phi \in C^3(\Omega)$ be a strictly convex function without critical points, and choose $\rho > 0$ such that
	\begin{align*}
	D^2 \phi (X, X) \ge \rho |X|^2, \quad X \in T_x \Omega,\ x \in \Omega.
	\end{align*}
	Let $\tau > 0$ and $w \in C^2(\Omega)$.
	We define $\ell = \tau \phi$, $v = e^\ell w$, and 
	\begin{align*}
	b = -(\Delta \ell) v - 2(\grad v, \grad \ell), \quad
	c = (|\grad v|^2 - (k^2 + |\nabla \ell|^2) v^2) \grad \ell.
	\end{align*}
	Then
	\begin{align*}
	e^{2 \tau\phi}
	\left( (a_0 \tau - b_0) \tau^2 w^2 + 
	(a_1 \tau - b_1)|\nabla w|^2 \right)
	+ \div(b \nabla v + c)
	\le e^{2 \tau\phi} (\Delta w + k^2 w)^2/2,
	\end{align*}
	where the constants $a_j, b_j > 0$, $j=0,1$, 
	depend only on $\rho$, $\inf\limits_{x \in \Omega} |\nabla \phi(x)|^2$ and $\sup\limits_{x \in \Omega} |\grad (\Delta \phi(x))|^2$.
\end{corollary}
\begin{proof}
	We employ the equality in \cref{lem_carleman_eq} 
	with $\ell = \tau \phi$ and $\sigma = \Delta \ell$.
	With this choice of $\sigma$, it holds that $a = 0$.
	As the two first terms on the right-hand side of the equality are positive, it is enough to consider 
	\begin{align*}
	&2 D^2 \ell(\grad v, \grad v) + 2 D^2 \ell (\grad \ell, \grad \ell) v^2 + R 
	\\&\quad\ge 
	2 \rho \tau |\grad v|^2 + 2 \rho \tau^3 |\grad \phi|^2 v^2
	- \tau |\grad (\Delta \phi)| |\grad v| |v|. 
	\end{align*}
	The claim follows by combining this with
	\begin{align*}
	|\grad v|^2 = e^{2\tau \phi} |\tau w \grad \phi + \grad w|^2
	\ge e^{2\tau \phi} \frac13  |\grad w|^2 - e^{2\tau \phi} \frac12 |\grad \phi|^2 \tau^2 w^2,
	\end{align*}
	and
	\begin{align*}
	\tau |\grad (\Delta \phi)| |\grad v| |v|
	\le C (|\grad v|^2 + \tau^2 |v|^2).
	\end{align*}
\end{proof}

The above Carleman estimate implies an inequality that is similar to the three-ball inequality, see e.g. \cite{Alessandrini}. The main difference is that here the foliation along spheres is followed in the opposite direction, i.e. the convex direction.

When continuing the solution inside the convex hull of $\omega$ as in \cite{Isakov}, we consider for simplicity a specific geometric setting defined in \cref{cor_Holder} below and illustrated in \cref{geometry}. The stability estimates we prove below in \cref{cor_Holder} and \cref{cor_Holder_impr}, and \cref{shifted3b} also hold in other geometric settings in which $B$ is included in the convex hull of $\omega$ and $B \setminus \omega$ does not touch the boundary of $\Omega$, such as the one in \cref{domain_conv}. We prove this in \cref{ex_convex_square}.

\begin{figure}[h]
	\centering
	\resizebox{0.5\textwidth}{!}{%
		\begin{tikzpicture}
		
		\filldraw[color=gray] (-2,0) arc [radius=2, start angle=180, end angle=360]
		-- (3,0) arc [radius=3, start angle=0, end angle=-180]
		-- cycle;
		
		\draw[dashed] (-5,0) -- (5,0);
		
		\filldraw (0,0) circle (2pt) node[align=center, above] {$0$};
		
		\draw (0,0) circle [radius=2];
		\draw (0,0) circle [radius=3];
		
		\draw (2.1,0) node[align=right, above] {$r$};
		\draw (3.1,0) node[align=right, above] {$R$};
		\draw (-4.5,0) node[align=right, below] {$H$};
		
		\filldraw [gray] (0,2.5) circle (2pt) node[align=center, above] {$y$};
		\draw (0,2.5) circle [radius=3.4]; 
		
		\draw (0,-2.75) node[align=right, above] {$\omega$};
		\draw (0,-1.75) node[align=right, above] {$B$};
		
		\end{tikzpicture}
	}%
	\caption{The geometric setting in \cref{cor_Holder}.}
	\label{geometry}
\end{figure}
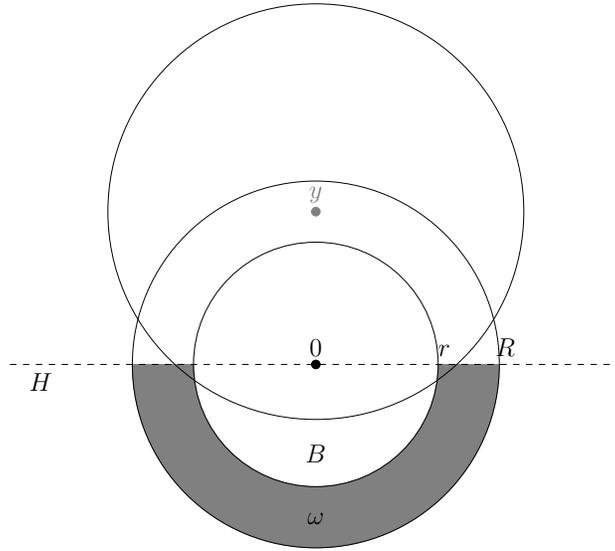

We use the following notation for a half space
\begin{align*}
H &= \{(x^0, \dots, x^n) \in \R^{1+n};\ x^0 < 0\}.
\end{align*}

\begin{corollary}
	\label{cor_Holder}
	Let $r > 0$, $\beta > 0$, $R > r$ and $\sqrt{r^2 + \beta^2} < \rho < \sqrt{R^2 + \beta^2}$. Define
	$y = (\beta, 0, \dots,0)$ and
	$$
	\Omega = H \cap B(0, R), 
	\quad
	\omega = \Omega \setminus \overline{B(0,r)}, 
	\quad
	B = \Omega \setminus \overline{B(y, \rho)}.
	$$
	Then there are $C > 0$ and $\alpha \in (0,1)$
	such that for all $u \in C^2(\Omega)$ and $k \ge 0$
	$$
	\norm{u}_{H^1(B)} 
	\le 
	C (\norm{u}_{H^1(\omega)} + 
	\norm{\Delta u + k^2 u}_{L^2(\Omega)})^\alpha \norm{u}_{H^1(\Omega)}^{1-\alpha}.
	$$
\end{corollary}
\begin{proof}
	Choose $\sqrt{r^2 + \beta^2} < s < \rho$
	and observe that $\p \Omega \setminus B(y,s) \subset \overline \omega$.
	Define $\phi(x) = |x-y|^2$.
	Then $\phi$ is smooth and strictly convex in $\overline{\Omega}$,
	and it does not have critical points there.
	
	Choose $\chi \in C_0^\infty(\Omega)$ such that 
	$\chi = 1$ in $\Omega \setminus (B(y,s) \cup \omega)$ and set $w = \chi u$.
	\cref{cor_ptwise_Carleman} implies that for large $\tau > 0$
	\begin{align}
	\label{Holder_step1}
	\int_{\Omega} (\tau^3 w^2 + \tau |\nabla w|^2) e^{2 \tau \phi} dx
	\le C \int_{\Omega} (\Delta w + k^2 w )^2 e^{2 \tau \phi} dx,
	\end{align}
	a result also stated, without a detailed proof, in \cite[Exercise 3.4.6]{IsakovBook}. The commutator $[\Delta, \chi]$ vanishes outside $B(y,s) \cup \omega$
	and $\phi < s^2$ in $B(y,s)$. 
	Hence the right-hand side of (\ref{Holder_step1})
	is bounded by a constant times
	\begin{align}
	\label{cylinders_step2}
	&\int_{\Omega} |\Delta u + k^2 u|^2 e^{2 \tau \phi} dx 
	+ \int_{B(y,s) \cup \omega} |[\Delta, \chi]u|^2 e^{2 \tau \phi} dx 
	\\\notag&\quad\le 
	Ce^{2 \tau (\beta+R)^2} (\norm{\Delta u + k^2 u}_{L^2(\Omega)}^2 
	+ \norm{u}_{H^1(\omega)}^2)
	+ C e^{2\tau s^2} 
	\norm{u}_{H^1(B(y,s))}^2.
	\end{align}	
	The left-hand side of (\ref{Holder_step1})
	is bounded  from below by
	\begin{align}
	\label{cylinders_step3}
	&\int_{B} 
	\left( \tau |\nabla u|^2 
	+ \tau^3 |u|^2 \right) e^{2 \tau \phi}\, dx 
	\ge 
	e^{2 \tau \rho^2} \norm{u}_{H^1(B)}^2.
	\end{align}
	The inequalities (\ref{Holder_step1})-(\ref{cylinders_step3}) imply
	$$
	\norm{u}_{H^1(B)}
	\le 
	e^{q\tau} 
	\left(\norm{\Delta u + k^2 u}_{L^2(\Omega)} + 
	\norm{u}_{H^1(\omega)} \right)
	+ e^{-p \tau} \norm{u}_{H^1(\Omega)},
	$$
	where $q=(\beta + R)^2 - \rho^2$ and $p = \rho^2 - s^2 > 0$.
	The claim follows from \cite[Lemma 5.2]{LeRousseau}.
\end{proof}

\begin{corollary}\label{cor_Holder_impr}
	Let $\omega \subset B \subset \Omega$ be defined as in \cref{cor_Holder}. Then there are $C>0$ and $\alpha\in(0,1)$ such that
	$$
	\norm{u}_{H^1(B)} 
	\le 
	C k (\norm{u}_{L^2(\omega)} + 
	\norm{\Delta u + k^2 u}_{H^{-1}(\Omega)})^\alpha (\norm{u}_{L^2(\Omega)} +  \norm{\Delta u + k^2 u}_{H^{-1}(\Omega)})^{1-\alpha}.
	$$
\end{corollary}
\begin{proof}
	Let $\omega_1 \subset \omega \subset B \subset \Omega_1 \subset \Omega$, denote for brevity by $\mathcal{L}$ the operator $\Delta + k^2$, and consider the following auxiliary problem
	\begin{align*}
	\mathcal{L}w &= \mathcal{L}u \quad \text{in } \Omega_1 \\
	\p_n w + \mathrm{i} kw &= 0 \quad \text{on } \p\Omega_1,
	\end{align*}
	whose solution satisfies the estimate \cite[Corollary 1.10]{Spence2016}
	$$
	\norm{\nabla w}_{L^2(\Omega_1)}+k\norm{w}_{L^2(\Omega_1)} \leq C k\norm{\mathcal{L}u}_{H^{-1}(\Omega_1)},
	$$
	which gives $$
	\norm{w}_{H^1(\Omega_1)} \leq C k \norm{\mathcal L u}_{H^{-1}(\Omega)}.
	$$
	For $v = u - w$ we have $\mathcal{L}v = 0$ in $\Omega_1$. The stability estimate in \cref{cor_Holder} used for $\omega_1,B,\Omega_1$ reads as
	$$
	\norm{v}_{H^1(B)}
	\le
	C \norm{v}^\alpha_{H^1(\omega_1)} \norm{v}_{H^1(\Omega_1)}^{1-\alpha},
	$$ 
	and the following estimates hold
	\begin{align*}
	\norm{u}_{H^1(B)}
	&\le
	\norm{v}_{H^1(B)} + \norm{w}_{H^1(B)}\\
	&\le
	C( \norm{u}_{H^1(\omega_1)} + \norm{w}_{H^1(\omega_1)} )^\alpha ( \norm{u}_{H^1(\Omega_1)} + \norm{w}_{H^1(\Omega_1)} )^{1-\alpha} + C k \norm{\mathcal{L}u}_{H^{-1}(\Omega)}
	\\
	&\le
	C( \norm{u}_{H^1(\omega_1)} + k \norm{\mathcal{L}u}_{H^{-1}(\Omega)} )^\alpha ( \norm{u}_{H^1(\Omega_1)} + k \norm{\mathcal{L}u}_{H^{-1}(\Omega)} )^{1-\alpha}.
	\end{align*}
	Now we choose a cutoff function $\chi \in C^{\infty}_0(\omega)$ such that $\chi = 1$ in $\omega_1$ and $\chi u$ satisfies
	$$
	\mathcal{L}(\chi u)=\chi \mathcal{L}u +[\mathcal{L},\chi]u,\quad \p_n (\chi u) + ik(\chi u) = 0 \text{ on } \p\omega.
	$$
	Since the commutator $[\mathcal{L},\chi]$ is of first order, using again \cite[Corollary 1.10]{Spence2016} we obtain
	
	\begin{align*}
	\norm{u}_{H^1(\omega_1)} &\le \norm{\chi u}_{H^1(\omega)}
	\le Ck \left( \norm{ [\mathcal{L},\chi]u }_{H^{-1}(\omega)} + \norm{ \chi \mathcal{L}u }_{H^{-1}(\omega)} \right) \\
	&\le Ck \left( \norm{u}_{L^2(\omega)} + \norm{\mathcal{L}u}_{H^{-1}(\omega)} \right)
	\end{align*}
	The same argument for $\Omega_1 \subset \Omega$ gives
	$$
	\norm{u}_{H^1(\Omega_1)} \le Ck( \norm{u}_{L^2(\Omega)} + \norm{\mathcal{L}u}_{H^{-1}(\Omega)}), 
	$$
	thus leading to the conclusion. 
\end{proof}

\subsection{Shifted three-ball inequality}
\label{sec_shifting}

\def\h{\hbar}
\def\scl{\text{scl}}

In this section we prove
an estimate as in \cref{cor_Holder},
but with the Sobolev indices shifted down one degree,
and our starting point is again the Carleman estimate 
in \cref{cor_ptwise_Carleman}.
When shifting Carleman estimates, as we want to keep track of the large parameter $\tau$, it is convenient to use the semiclassical version of pseudodifferential calculus. 
We write $\h > 0$ for the semiclassical parameter
that satisfies $\h = 1/\tau$.

The semiclassical (pseudo)differential operators are (pseudo)differential operators where, roughly speaking, each derivative is multiplied
by $\h$, for the precise definition see Section 4.1 of \cite{Zworski}.
The scale of semiclassical Bessel potentials is defined by 
$$
J^s = (1-\h^2 \Delta)^{s/2}, \quad s \in \R,
$$
and the semiclassical Sobolev spaces by
$$
\norm{u}_{H_\scl^s(\R^n)} = \norm{J^s u}_{L^2(\R^n)}.
$$
Then a semiclassical differential operator of order $m$
is continuous from $H_\scl^{m+s}(\R^n)$ to $H_\scl^s(\R^n)$,
see e.g. Section 8.3 of \cite{Zworski}.

We will give a shifting argument that is similar to that in Section 4 of \cite{DKSU}.
To this end, we need the following pseudolocal and commutator estimates for semiclassical pseudodifferential operators,
see e.g. (4.8) and (4.9) of \cite{DKSU}.
Suppose that $\psi,\chi \in C_0^\infty(\R^n)$
and that $\chi = 1$ near $\supp(\psi)$,
and let $A, B$ be two semiclassical pseudodifferential 
operators of orders $s, m$, respectively.
Then for all $p,q,N \in \R$, there is $C>0$
\begin{align}
\norm{(1-\chi) A (\psi u)}_{H_\scl^{p}(\R^n)}
&\le C \h^N \norm{u}_{H_\scl^{q}(\R^n)}, \label{pseudolocality}
\\
\norm{[A,B] u}_{H_\scl^{p}(\R^n)} 
&\le C \h \norm{u}_{H_\scl^{p+s+m-1}(\R^n)}. \label{commutator}
\end{align}
Both these estimates follow from the composition calculus, see e.g. \cite[Theorem 4.12]{Zworski}.

Let $\phi$ be as in \cref{cor_ptwise_Carleman}
and set $\ell = \phi / \h$ and $\sigma = \Delta \ell$ in \cref{lem_carleman_eq}. Then 
\begin{align*}
(e^{\phi / \h} \Delta e^{-\phi / \h} v + k^2 v)^2 / 2
&\ge 
2 \h^{-1} D^2 \phi(\nabla v, \nabla v) 
+ 2 \h^{-3} D^2 \phi(\nabla \phi, \nabla \phi) v^2
\\&\quad
+ \div(b \nabla v + B) + \h^{-1} (\nabla \Delta \phi, \nabla v) v
\end{align*}
Write $P = e^{\phi / \h} \h^2 \Delta e^{-\phi / \h}$
and let $v \in C_0^\infty(\Omega')$ where $\Omega' \subset \R^n$ is open and bounded, and $\overline \Omega \subset \Omega'$. Then, rescaling by $\h^4$,
\begin{align*}
C \norm{P v + \h^2 k^2 v}_{L^2(\R^n)}^2
\ge \h \norm{\h \nabla v}_{L^2(\R^n)}^2 + \h \norm{v}_{L^2(\R^n)}^2 - C \h^2 \norm{v}_{H_\scl^1(\R^n)}^2,
\end{align*}
and for small enough $\h > 0$ we obtain
\begin{align*}
\sqrt \h \norm{v}_{H_\scl^1(\R^n)} \le C \norm{P v + \h^2 k^2 v}_{L^2(\R^n)}.
\end{align*} 
Now the conjugated operator $P$ is a semiclassical differential operator,
\begin{align*}
Pu = e^{\phi / \h} \h^2 \div \grad (e^{-\phi / \h} u)
= \h^2 \Delta u - 2 (\nabla \phi, \h \nabla u)
- \h (\Delta \phi) u + |\nabla \phi|^2 u. 
\end{align*}

Let $\chi, \psi \in C_0^\infty(\Omega')$ and suppose that $\psi=1$ near $\Omega$ and $\chi = 1$ near $\supp(\psi)$. Then for $v \in C_0^\infty(\Omega)$,
$$
\norm{v}_{H_\scl^{1+s}(\R^n)}
\le \norm{\chi J^s v}_{H_\scl^{1}(\R^n)} 
+ \norm{(1-\chi) J^s \psi v}_{H_\scl^{1}(\R^n)}
\le C \norm{\chi J^s v}_{H_\scl^{1}(\R^n)} 
$$
where we used the pseudolocality (\ref{pseudolocality}) to absorb the second term 
on the right-hand side by the left-hand side.  
We have 
\begin{align}\label{shift_step1}
\sqrt \h \norm{v}_{H_\scl^{1+s}(\R^n)}\le C\sqrt \h \norm{\chi J^s v}_{H_\scl^1(\R^n)} 
\le C \norm{(P + \h^2 k^2) \chi J^s v}_{L^2(\R^n)},
\end{align}
and using the commutator estimate (\ref{commutator}), we have
\begin{align*}
\norm{[P,\chi J^s] v}_{L^2(\R^n)}
\le C \h \norm{v}_{H^{1+s}_\scl(\R^n)}.
\end{align*}
This can be absorbed by the left-hand side of (\ref{shift_step1}). Thus
\begin{align*}
\sqrt \h \norm{v}_{H_\scl^{1+s}(\R^n)}\le C 
\norm{\chi J^s (P + \h^2 k^2) v}_{L^2(\R^n)}
\le C \norm{(P + \h^2 k^2) v}_{H_\scl^{s}(\R^n)}.
\end{align*}

Take now $s=-1$ and let the cutoff $\chi$ and the weight $\phi$ be as in the proof of \cref{cor_Holder}, with the additional condition on $\chi$ such that there is 
$\psi \in C_0^\infty(B(y,s) \cup \omega)$
satisfying $\psi = 1$ in $\supp([P,\chi])$.

Let $u \in C^\infty(\R^n)$ and set $w = e^{\phi/\h} u$. Then the previous estimate becomes
\begin{align*}
\sqrt \h \norm{\chi w}_{L^2(\R^n)}
\le C \norm{(P + \h^2 k^2) \chi w}_{H_\scl^{-1}(\R^n)}.
\end{align*}
We have 
$$
\norm{[P,\chi] w}_{H_\scl^{-1}(\R^n)}
=
\norm{[P,\chi] \psi w}_{H_\scl^{-1}(\R^n)}
\le C \h \norm{\psi w}_{L^2 (\R^n)}.
$$
Using the norm inequality $\norm{\cdot}_{H_\scl^{-1}(\R^n)} \le C \h^{-2} \norm{\cdot}_{H^{-1}(\R^n)}$, we thus obtain
\begin{align*}
\sqrt \h \norm{\chi e^{\phi/\h} u}_{L^2(\R^n)}
&\le 
C \norm{\chi(e^{\phi / \h} \Delta e^{-\phi / \h} + k^2) w}_{H_\scl^{-1}(\R^n)} + C \h \norm{\psi w}_{L^2 (\R^n)}
\\&\le
C \h^{-2} \norm{\chi e^{\phi/\h} (\Delta u + k^2 u)}_{H^{-1}(\R^n)} + C \h \norm{\psi e^{\phi/\h} u}_{L^2 (\R^n)}
\end{align*}
Using the same notation as in the proof of \cref{cor_Holder}, due to the choice of $\psi$ we get
$$
e^{\rho^2/\h} \norm{u}_{L^2(B)}
\le 
C e^{(\beta+R)^2/\h} 
\left( \h^{-\frac72} \norm{\Delta u + k^2 u}_{H^{-1}(\Omega)} + \h^{\frac12} \norm{u}_{L^2(\omega)} \right)
+ C e^{s^2/\h} \h^{\frac12} \norm{u}_{L^2(\Omega)},
$$
for small enough $\h > 0$. Absorbing the negative power of $\h$ in the exponential, and using \cite[Lemma 5.2]{LeRousseau}, we conclude the proof of the following result.

\begin{lemma}\label{shifted3b}
	Let $\omega \subset B \subset \Omega$ be defined as in \cref{cor_Holder}. Then there are $C>0$ and $\alpha\in(0,1)$ such that
	\begin{equation*}
	\norm{u}_{L^2(B)} 
	\le 
	C (\norm{u}_{L^2(\omega)} + 
	\norm{\Delta u + k^2 u}_{H^{-1}(\Omega)})^\alpha \norm{u}_{L^2(\Omega)}^{1-\alpha}.
	\end{equation*}
\end{lemma}

\section{Stabilized finite element method}\label{stabilizedfem}
We aim to solve the unique continuation problem for the Helmholtz equation
\begin{equation}\label{helmholtz}
\Delta u + k^2u = -f \text{ in } \Omega, \quad u=q|_{\omega},
\end{equation}
where $\omega \subset \Omega \subset \R^{1+n}$ are open, $f\in H^{-1}(\Omega)$ and $q\in L^2(\omega)$ are given.
Following the optimization based approach in \cite{BurmanSIAM, BHL2018} we will make use of the continuum stability estimates in \cref{continuum_estimates} when deriving error estimates for the finite element approximation.
\subsection{Discretization}
Consider a family $\mathcal{T}=\{\mathcal{T}_h\}_{h>0}$ of triangulations of $\Omega$ consisting of simplices such that the intersection of any two distinct ones is either a common vertex, a common edge or a common face. Also, assume that the family $\mathcal{T}$ is quasi-uniform. 
Let
$$V_h = \{u\in C(\bar{\Omega}): u|_K \in \P_1(K), K\in \mathcal{T}_h\}$$
be the $H^1$-conformal approximation space based on the $\P_1$ finite element and let
$$
W_h = V_h \cap H^1_0(\Omega).
$$
Consider the orthogonal $L^2$-projection $\Pi_h:L^2(\Omega) \to V_h$, which  satisfies
\begin{align*}
(u-\Pi_h u,v)_{L^2(\Omega)} &= 0,\quad u\in L^2(\Omega),\, v\in V_h, \\
\norm{\Pi_h u}_{L^2(\Omega)} &\le \norm{u}_{L^2(\Omega)},\quad u\in L^2(\Omega),
\end{align*}
and the Scott-Zhang interpolator $\pi_h:H^1(\Omega) \to V_h$, that preserves vanishing Dirichlet boundary conditions. Both operators have the following stability and approximation properties, see e.g. \cite[Chapter 1]{ErnBook},
\begin{align}
\norm{i_h u}_{H^1(\Omega)} &\le C \norm{u}_{H^1(\Omega)}, &u\in H^1(\Omega), \label{H1stab} \\
\norm{u-i_h u}_{H^m(\Omega)} &\le C h^{k-m} \norm{u}_{H^k(\Omega)}, &u\in H^k(\Omega) \label{interp},
\end{align}
where $i= \pi,\Pi,\, k=1,2$ and $m=0,k-1$.

The regularization on the discrete level will be based on the $L^2$-control of the gradient jumps over elements edges using the jump stabilizer
\begin{equation*}
\J (u,u) = \sum_{F\in\mathcal{F}_h} \int_F h \llbracket n \cdot \nabla u \rrbracket ^2 ds,\quad u\in V_h,
\end{equation*}
where $\mathcal{F}_h$ is the set of all internal faces, and the jump over $F\in\mathcal{F}_h$ is given by
$$
\llbracket n \cdot \nabla{u} \rrbracket_F = n_1 \cdot \nabla{u}|_{K_1} + n_2 \cdot \nabla{u}|_{K_2}, 
$$
with $K_1,K_2\in \mathcal{T}_h$ being two simplices such that $K_1 \cap K_2 = F$, and $n_j$ the outward normal of $K_j,\, j=1,2$. The face subscript is omitted when there is no ambiguity.

\begin{lemma}
	There is $C>0$ such that all $u\in V_h,\, v\in H^1_0(\Omega),\, w\in H^2(\Omega)$ and $h>0$ satisfy
	\begin{align}
	(\nabla u,\nabla v)_{L^2(\Omega)} &\le C \J(u,u)^{1/2} (h^{-1}\norm{v}_{L^2(\Omega)} + \norm{v}_{H^1(\Omega)}),\label{jumpineq1} \\
	\J(i_h w,i_h w) &\le C h^2 \norm{w}^2_{H^2(\Omega)},\quad i\in \{\pi,\Pi\}.\label{jumpineq2}
	\end{align}
	\begin{proof}
		See \cite[Lemma 2]{BurmanOksanen} when the interpolator is $\pi_h$. Since this proof uses just the approximation properties of $\pi_h$, it holds verbatim for $\Pi_h$.
	\end{proof}
\end{lemma}

Adopting the notation
$$
a(u,z)=(\nabla u,\nabla z)_{L^2(\Omega)},\quad G_f(u,z)=a(u,z)-k^2(u,z)_{L^2(\Omega)} - \pair{f,z},\quad G=G_0,
$$
we write for $u\in H^1(\Omega)$ the weak formulation of $\Delta u + k^2u = -f$ as
$$
G_f(u,z)=0, \quad z\in H^1_0(\Omega).
$$

Our approach is to find the saddle points of the Lagrangian functional
\begin{equation*}
L_{q,f}(u,z) = \frac{1}{2} \norm{u-q}^2_{\omega} + \frac{1}{2}s(u,u)-\frac{1}{2}s^*(z,z)+G_f(u,z),
\end{equation*}
where $\norm{\cdot}_\omega$ denotes $\norm{\cdot}_{L^2(\omega)}$, and $s$ and $s^*$ are stabilizing (regularizing) terms for the primal and dual variables that should be consistent and vanish at optimal rates. The stabilization must control certain residual quantities representing the data of the error equation. The primal stabilizer will be based on the continuous interior penalty given by $\J$. It must take into account the zeroth order term of the Helmholtz operator. The dual variable can be stabilized in the $H^1$-seminorm. Notice that when the PDE-constraint is satisfied, $z=0$ is the solution for the dual variable of the saddle point, thus the stabilizer $s^*$ is consistent. Hence we make the following choice
$$
s(u,u) = \J(u,u) + \norm{hk^2 u}^2_{L^2(\Omega)},\quad s^* = a.
$$ For a detailed presentation of such discrete stabilizing operators we refer the reader to \cite{BurmanSIAM} or \cite{BurmanChapter}.
We define on $V_h$ and $W_h$, respectively, the norms
$$
\norm{u}_V = s(u,u)^{1/2}, \quad u\in V_h, \quad \norm{z}_W = s^*(z,z)^{1/2},\quad z\in W_h,
$$
together with the norm on $V_h \times W_h$ defined by
$$
\tnorm{(u,z)}^2 = \norm{u}^2_V + \norm{u}^2_{\omega} + \norm{z}^2_W.
$$
The saddle points $(u,z)\in V_h \times W_h$ of the Lagrangian $L_{q,f}$ satisfy
\begin{equation}\label{weakform}
A[(u,z),(v,w)] = (q,v)_{\omega} + \pair{f,w},\quad (v,w)\in V_h \times W_h,
\end{equation}
where $A$ is the symmetric bilinear form
\begin{equation*}
A[(u,z),(v,w)] = (u,v)_{\omega} + s(u,v) + G(v,z) - s^*(z,w) + G(u,w).
\end{equation*}
Since $A[(u,z),(u,-z)] = \norm{u}^2_{\omega} + \norm{u}^2_V + \norm{z}^2_W$ we have the following inf-sup condition
\begin{equation}\label{inf_sup}
\sup_{(v,w)\in V_h \times W_h} \frac{A[(u,z),(v,w)]}{\tnorm{(v,w)}} \ge \tnorm{(u,z)}
\end{equation}
that guarantees a unique solution in $V_h \times W_h$ for (\ref{weakform}).

\subsection{Error estimates}

We start by deriving some lower and upper bounds for the norm $\norm{\cdot}_V$. For $u_h \in V_h,\, z\in H^1_0(\Omega)$, we use (\ref{jumpineq1}) to bound
\begin{align*}
G(u_h,z) &= (\nabla u_h,\nabla z)_{L^2(\Omega)} - k^2(u_h,z)_{L^2(\Omega)} \\
&\le C \J(u_h,u_h)^{1/2} (h^{-1}\norm{z}_{L^2(\Omega)} + \norm{z}_{H^1(\Omega)}) + k^2 \norm{u_h}_{L^2(\Omega)} \norm{z}_{L^2(\Omega)},
\end{align*}
and hence
\begin{equation}\label{lowerv}
G(u_h,z) \le C \norm{u_h}_V (h^{-1}\norm{z}_{L^2(\Omega)} + \norm{z}_{H^1(\Omega)}).
\end{equation}
For $u \in H^2(\Omega)$, from (\ref{jumpineq2}) and the stability of the $L^2$-projection
\begin{equation*}
\norm{\Pi_h u}^2_V = \J(\Pi_h u,\Pi_h u) + \norm{hk^2 \Pi_h u}^2_{L^2(\Omega)} \le C(h^2 \norm{u}^2_{H^2(\Omega)} + \norm{hk^2 u}^2_{L^2(\Omega)}) 
\end{equation*}
implies
\begin{equation}\label{upperv}
\norm{\Pi_h u}_V \le C h (\norm{u}_{H^2(\Omega)} + k^2 \norm{u}_{L^2(\Omega)}) = C h \norm{u}_*, 
\end{equation}
where $\norm{u}_*$ is defined as in (\ref{starnorm}). 

\begin{lemma}\label{tnormerror}
	Let $u\in H^2(\Omega)$ be the solution to (\ref{helmholtz}) and $(u_h,z_h)\in V_h \times W_h$ be the solution to (\ref{weakform}). Then there exists $C>0$ such that for all $h\in(0,1)$
	$$
	\tnorm{(u_h-\Pi_h u,z_h)} \le C h \norm{u}_*.
	$$
	\begin{proof}
		Due to the inf-sup condition (\ref{inf_sup}) it is enough to prove that for $(v,w)\in V_h \times W_h$,
		$$
		A[(u_h-\Pi_h u,z_h),(v,w)] \le C h \norm{u}_* \tnorm{(v,w)}.
		$$
		The weak form of (\ref{helmholtz}) implies that
		\begin{equation*}
		A[(u_h-\Pi_h u,z_h),(v,w)] = (u-\Pi_h u,v)_\omega + G(u-\Pi_h u,w) - s(\Pi_h u,v).
		\end{equation*}
		Using (\ref{interp}) we bound the first term to get
		$$
		(u-\Pi_h u,v)_\omega \le C h^2 \norm{u}_{H^2(\Omega)} \norm{v}_\omega.
		$$
		For the second term we use the $L^2$-orthogonality property of $\Pi_h$, and (\ref{interp}) to obtain
		\begin{equation*}
		G(u-\Pi_h u,w) = (\nabla (u-\Pi_h u), \nabla w)_{L^2(\Omega)}
		\le C h \norm{w}_W \norm{u}_{H^2(\Omega)},
		\end{equation*}
		while for the last term we employ (\ref{upperv}) to estimate
		$$
		s(\Pi_h u,v) \le \norm{\Pi_h u}_V \norm{v}_V \leq C h \norm{u}_* \norm{v}_V.
		$$
	\end{proof}
\end{lemma}

\begin{theorem}\label{L2error}
	Let $\omega \subset B \subset \Omega$ be defined as in \cref{cor_Holder}. Let $u\in H^2(\Omega)$ be the solution to (\ref{helmholtz}) and $(u_h,z_h)\in V_h \times W_h$ be the solution to (\ref{weakform}). Then there are $C>0$ and $\alpha \in (0,1)$ such that for all $k,\, h>0$ with $k h \lesssim 1$
	$$
	\norm{u-u_h}_{L^2(B)} \le C (h k)^\alpha k^{\alpha-2} \norm{u}_*.
	$$
	
	\begin{proof}
		Consider the residual $\pair{r,w}=G(u_h-u,w)=G(u_h,w)-\pair{f,w},\, w\in H^1_0(\Omega)$. Taking $v=0$ in (\ref{weakform}) we get $G(u_h,w) = \pair{f,w} + s^*(z_h,w),\, w\in W_h$ which implies that
		\begin{align*}
		\pair{r,w} &= G(u_h,w) - \pair{f,w} - G(u_h,\pi_h w) + G(u_h,\pi_h w) \\
		&= G(u_h,w-\pi_h w) - \pair{f,w-\pi_h w} + s^*(z_h,\pi_h w),\quad w\in H^1_0(\Omega).
		\end{align*}
		Using (\ref{lowerv}) and (\ref{interp}) we estimate the first term
		\begin{align*}
		G(u_h,w-\pi_h w) &\le C \norm{u_h}_V (h^{-1} \norm{w-\pi_h w}_{L^2(\Omega)} + \norm{w-\pi_h w}_{H^1(\Omega)}) \\
		&\le C \norm{u_h}_V \norm{w}_{H^1(\Omega)} \le C h \norm{u}_* \norm{w}_{H^1(\Omega)},
		\end{align*}
		since, due to \cref{tnormerror} and (\ref{upperv})
		\begin{equation*}
		\norm{u_h}_V \le \norm{u_h-\Pi_h u}_V + \norm{\Pi_h u}_V
		\le C h \norm{u}_*.
		\end{equation*}
		The second term is bounded by using (\ref{interp})
		$$
		\pair{f,w-\pi_h w} \le \norm{f}_{L^2(\Omega)} \norm{w-\pi_h w}_{L^2(\Omega)} \le C h \norm{f}_{L^2(\Omega)} \norm{w}_{H^1(\Omega)}
		$$
		and the last term by using \cref{tnormerror} and the $H^1$-stability (\ref{H1stab})
		$$
		s^*(z_h,\pi_h w) \le \norm{z_h}_W \norm{\pi_h w}_W \le C h \norm{u}_* \norm{w}_{H^1(\Omega)}.
		$$
		Hence the following residual norm estimate holds
		$$
		\norm{r}_{H^{-1}(\Omega)} \le C h (\norm{u}_* + \norm{f}_{L^2(\Omega)})  \le C h \norm{u}_*.
		$$
		Using the continuum estimate in \cref{shifted3b} for $u-u_h$ we obtain the following error estimate
		$$
		\norm{u-u_h}_{L^2(B)} \le C (\norm{u-u_h}_{L^2(\omega)} + \norm{r}_{H^{-1}(\Omega)})^\alpha \norm{u-u_h}_{L^2(\Omega)}^{1-\alpha}.
		$$
		By (\ref{interp}) and \cref{tnormerror} we have the bounds
		\begin{align*}
		\norm{u-u_h}_{L^2(\omega)} &\le \norm{u-\Pi_h u}_{L^2(\omega)} + \norm{u_h-\Pi_h u}_{L^2(\omega)} \\
		&\le C h \norm{u}_{H^1(\Omega)} + C h \norm{u}_* \\
		&\le  C h \norm{u}_*
		\end{align*}
		and
		\begin{align*}
		\norm{u-u_h}_{L^2(\Omega)} &\le \norm{u-\Pi_h u}_{L^2(\Omega)} + \norm{u_h-\Pi_h u}_{L^2(\Omega)} \\
		&\le C h^2 \norm{u}_{H^2(\Omega)} + C h^{-1} k^{-2} \norm{u_h-\Pi_h u}_V \\
		&\le C ((h^2+k^{-2})\norm{u}_{H^2(\Omega)} + \norm{u}_{L^2(\Omega)}) \\
		&\le C k^{-2} \norm{u}_*
		\end{align*}
		thus leading to the conclusion.
	\end{proof}
\end{theorem}

\begin{theorem}\label{H1error}
	Let $\omega \subset B \subset \Omega$ be defined as in \cref{cor_Holder}. Let $u\in H^2(\Omega)$ be the solution to (\ref{helmholtz}) and $(u_h,z_h)\in V_h \times W_h$ be the solution to (\ref{weakform}). Then there are $C>0$ and $\alpha \in (0,1)$ such that for all $k,\, h>0$ with $k h \lesssim 1$
	$$
	\norm{u-u_h}_{H^1(B)} \le C (h k)^\alpha \norm{u}_*.
	$$
\end{theorem}
\begin{proof}
	We employ a similar argument as in the proof of \cref{L2error} with the same estimates for the residual norm and the $L^2$-errors in $\omega$ and $\Omega$, only now using the continuum estimate in \cref{cor_Holder_impr} to obtain
	\begin{align*}
	\norm{u-u_h}_{H^1(B)} &\le C k ( \norm{u-u_h}_{L^2(\omega)} + \norm{r}_{H^{-1}(\Omega)})^\alpha ( \norm{u-u_h}_{L^2(\Omega)} + \norm{r}_{H^{-1}(\Omega)})^{1-\alpha}
	\\
	&\le C k h^\alpha (k^{-2}+h)^{1-\alpha} \norm{u}_*,
	\end{align*}
	which ends the proof.
\end{proof}

Let us remark that if we make the assumption $k^2 h \lesssim 1$ then the estimate in \cref{H1error} becomes
$$
\norm{u-u_h}_{H^1(B)} \le C (h k^2)^\alpha k^{-1} \norm{u}_*,
$$
and combining \cref{L2error} and \cref{H1error} we obtain the following result.

\begin{corollary}
	Let $\omega \subset B \subset \Omega$ be defined as in \cref{cor_Holder}. Let $u\in H^2(\Omega)$ be the solution to (\ref{helmholtz}) and $(u_h,z_h)\in V_h \times W_h$ be the solution to (\ref{weakform}). Then there are $C>0$ and $\alpha \in (0,1)$ such that for all $k,\, h>0$ with $k^2 h \lesssim 1$
	$$
	k\norm{u-u_h}_{L^2(B)} + \norm{u-u_h}_{H^1(B)} \le C (h k^2)^\alpha k^{-1} \norm{u}_*.
	$$
\end{corollary}

Comparing with the well-posed boundary value problem (\ref{helmholtzbvp}) and the sharp bounds (\ref{well-posed1}) and (\ref{well-posed2}), we note that the $k^{-1} \norm{u}_*$ term in the above estimate is analogous to the well-posed case term $\norm{f}_{L^2(\Omega)}$.

\subsection{Data perturbations} The analysis above can also handle the perturbed data
$$
\tilde{q} = q + \delta q,\quad \tilde{f} = f + \delta f, 
$$
with the unperturbed data $q,\, f$ in (\ref{helmholtz}), and perturbations $\delta q \in L^2(\omega),\, \delta f \in H^{-1}(\Omega)$ measured by
$$
\delta (\tilde{q},\tilde{f}) = \norm{\delta q}_\omega + \norm{\delta f}_{H^{-1}(\Omega)}.
$$
The saddle points $(u,z)\in V_h \times W_h$ of the perturbed Lagrangian $L_{\tilde{q},\tilde{f}}$ satisfy
\begin{equation}\label{pweakform}
A[(u,z),(v,w)] = (\tilde{q},v)_{\omega} + \pair{\tilde{f},w},\quad (v,w)\in V_h \times W_h.
\end{equation}

\begin{lemma}\label{ptnormerror}
	Let $u\in H^2(\Omega)$ be the solution to the unperturbed problem (\ref{helmholtz}) and $(u_h,z_h)\in V_h \times W_h$ be the solution to the perturbed problem (\ref{pweakform}). Then there exists $C>0$ such that for all $h\in(0,1)$
	$$
	\tnorm{(u_h-\Pi_h u,z_h)} \le C (h \norm{u}_* +\delta(\tilde{q},\tilde{f})).
	$$
	\begin{proof}
		Proceeding as in the proof of \cref{tnormerror}, the weak form gives
		\begin{align*}
		A[(u_h-\Pi_h u,z_h),(v,w)] &= (u-\Pi_h u,v)_\omega + G(u-\Pi_h u,w) - s(\Pi_h u,v)
		\\
		&+ (\delta q,v)_\omega + \pair{\delta f,w}.
		\end{align*}
		We bound the perturbation terms by
		\begin{align*}
		(\delta q,v)_\omega + \pair{\delta f,w} &\le \norm{\delta q}_\omega \norm{v}_\omega + C\norm{\delta f}_{H^{-1}(\Omega)} \norm{w}_W \\
		&\le C \delta(\tilde{q},\tilde{f}) \tnorm{(v,w)}
		\end{align*}
		and we conclude by using the previously derived bounds for the other terms.
	\end{proof}
\end{lemma}

\begin{theorem}\label{L2error_perturbed}
	Let $\omega \subset B \subset \Omega$ be defined as in \cref{cor_Holder}. Let $u\in H^2(\Omega)$ be the solution to the unperturbed problem (\ref{helmholtz}) and $(u_h,z_h)\in V_h \times W_h$ be the solution to the perturbed problem (\ref{pweakform}). Then there are $C>0$ and $\alpha \in (0,1)$ such that for all $k,\, h>0$ with $k h \lesssim 1$
	$$
	\norm{u-u_h}_{L^2(B)} \le C (h k)^\alpha k^{\alpha-2} ( \norm{u}_* +  h^{-1} \delta(\tilde{q},\tilde{f}) ).
	$$
	\begin{proof}
		Following the proof of \cref{L2error}, the residual satisfies
		$$
		\pair{r,w} = G(u_h,w-\pi_h w) - \pair{f,w-\pi_h w} + s^*(z_h,\pi_h w) + \pair{\delta f, \pi_h w},\quad w\in H^1_0(\Omega)
		$$
		and
		$$
		\norm{r}_{H^{-1}(\Omega)} \le C (\norm{u_h}_V + h\norm{f}_{L^2(\Omega)} + \norm{z_h}_W + \norm{\delta f}_{H^{-1}(\Omega)}).
		$$
		Bounding the first term in the right-hand side by \cref{ptnormerror} and (\ref{upperv})
		\begin{equation*}
		\norm{u_h}_V \le \norm{u_h-\Pi_h u}_V + \norm{\Pi_h u}_V
		\le C (h \norm{u}_* +\delta(\tilde{q},\tilde{f}))
		\end{equation*}
		and the third one by \cref{ptnormerror} again, we obtain
		$$
		\norm{r}_{H^{-1}(\Omega)} \le C h(\norm{u}_* + \norm{f}_{L^2(\Omega)}) + C \delta(\tilde{q},\tilde{f}) \le C (h\norm{u}_* + \delta(\tilde{q},\tilde{f})).
		$$
		The continuum estimate in \cref{shifted3b} applied to $u-u_h$ gives
		$$
		\norm{u-u_h}_{L^2(B)} \le C \left( h \norm{u}_* + \delta(\tilde{q},\tilde{f}) \right)^\alpha \norm{u-u_h}_{L^2(\Omega)}^{1-\alpha},
		$$
		where $\norm{u-u_h}_{L^2(\omega)}$ was bounded by using \cref{ptnormerror} and (\ref{interp}). Then the bound
		\begin{align*}
		\norm{u-u_h}_{L^2(\Omega)} &\le \norm{u-\Pi_h u}_{L^2(\Omega)} + \norm{u_h-\Pi_h u}_{L^2(\Omega)} \\
		&\le C ( h^2 \norm{u}_{H^2(\Omega)} + h^{-1} k^{-2} \norm{u_h-\Pi_h u}_V )\\
		&\le C ( h^2 \norm{u}_{H^2(\Omega)} + k^{-2} \norm{u}_* + h^{-1} k^{-2} \delta(\tilde{q},\tilde{f}) ) \\
		&\le C k^{-2} (\norm{u}_* + h^{-1} \delta(\tilde{q},\tilde{f}))
		\end{align*}
		concludes the proof.
	\end{proof}
\end{theorem}

\begin{theorem}\label{H1error_perturbed}
	Let $\omega \subset B \subset \Omega$ be defined as in \cref{cor_Holder}. Let $u\in H^2(\Omega)$ be the solution to the unperturbed problem (\ref{helmholtz}) and $(u_h,z_h)\in V_h \times W_h$ be the solution to the perturbed problem (\ref{pweakform}). Then there are $C>0$ and $\alpha \in (0,1)$ such that for all $k,\, h>0$ with $k h \lesssim 1$
	$$
	\norm{u-u_h}_{H^1(B)} \le C (h k)^\alpha ( \norm{u}_* +  h^{-1} \delta(\tilde{q},\tilde{f}) ).
	$$
	\begin{proof}
		Following the proof of \cref{L2error_perturbed}, we now use \cref{cor_Holder_impr} to derive
		\begin{align*}
		\norm{u-u_h}_{H^1(B)} &\le C k ( \norm{u-u_h}_{L^2(\omega)} + \norm{r}_{H^{-1}(\Omega)})^\alpha ( \norm{u-u_h}_{L^2(\Omega)} + \norm{r}_{H^{-1}(\Omega)})^{1-\alpha}
		\\
		&\le C k \left(h\norm{u}_* + \delta(\tilde{q},\tilde{f})\right)^{\alpha} \left((k^{-2}+h)(\norm{u}_* +  h^{-1} \delta(\tilde{q},\tilde{f}))\right)^{1-\alpha}
		\\
		&\le C k h^\alpha (k^{-2}+h)^{1-\alpha} (\norm{u}_* +  h^{-1} \delta(\tilde{q},\tilde{f})),
		\end{align*}
		which ends the proof.
	\end{proof}
\end{theorem}
Analogous to the unpolluted case, if $k^2 h \lesssim 1$ the above result becomes
$$
\norm{u-u_h}_{H^1(B)} \le C (h k^2)^\alpha k^{-1} (\norm{u}_* +  h^{-1} \delta(\tilde{q},\tilde{f})),
$$
and combining \cref{L2error_perturbed} and \cref{H1error_perturbed} gives the following.

\begin{corollary}
	Let $\omega \subset B \subset \Omega$ be defined as in \cref{cor_Holder}. Let $u\in H^2(\Omega)$ be the solution to the unperturbed problem (\ref{helmholtz}) and $(u_h,z_h)\in V_h \times W_h$ be the solution to the perturbed problem (\ref{pweakform}). Then there are $C>0$ and $\alpha \in (0,1)$ such that for all $k,\, h>0$ with $k^2 h \lesssim 1$
	$$
	k\norm{u-u_h}_{L^2(B)} + \norm{u-u_h}_{H^1(B)} \le C (h k^2)^\alpha k^{-1} (\norm{u}_* +  h^{-1} \delta(\tilde{q},\tilde{f})).
	$$
\end{corollary}

\section{Numerical examples}

We illustrate the above theoretical results for the unique continuation problem (\ref{helmholtz}) with some numerical examples. Drawing on previous results in \cite{BurmanSIAM}, we adjust the  stabilizer in (\ref{weakform}) with a fixed stabilization parameter $\gamma>0$ such that $s(u,v) = \gamma \J(u,v) + \gamma h^2k^4 (u,v)_{L^2(\Omega)}$. The error analysis stays unchanged under this rescaling. Various numerical experiments indicate that $\gamma = 10^{-5}$ is a near-optimal value for different kinds of geometries and solutions. The implementation of our method and all the computations have been carried out in FreeFem++ \cite{FreeFem}. The domain $\Omega$ is the unit square, and the triangulation is uniform with alternating left and right diagonals, as shown in \cref{mesh}. The mesh size is taken as the inverse square root of the number of nodes.

In the light of the convexity assumptions in \cref{continuum_estimates}, we shall consider two different geometric settings: one in which the data is continued in the convex direction, inside the convex hull of $\omega$, and one in which the solution is continued in the non-convex direction, outside the convex hull of $\omega$.

In the convex setting, given in \cref{domain_conv}, we take
\begin{equation}\label{convex}
\omega = \Omega \setminus [0.1,0.9] \times [0.25,1],\quad
B = \Omega \setminus [0.1,0.9] \times [0.95,1]
\end{equation}
for continuing the solution inside the convex hull of $\omega$. This example does not correspond exactly to the specific geometric setting in \cref{cor_Holder}, but all the theoretical results are valid in this case as proven in \cref{ex_convex_square} below.

\begin{example}\label{ex_convex_square}
	Let $\omega\subset B \subset \Omega$ be defined by \eqref{convex} (\cref{domain_conv}). Then the stability estimates in \cref{cor_Holder}, \cref{cor_Holder_impr} and \cref{shifted3b} hold true.
	\begin{proof}
		Consider an extended rectangle $\tilde{\Omega} \supset \Omega$ such that the unit square $\Omega$ is centred horizontally and touches the upper side of $\tilde{\Omega}$, and $\tilde{\omega} \supset \omega$ and $\tilde{B} \supset B$ are defined as in \cref{cor_Holder}. Choose a smooth cutoff function $\chi$ such that $\chi=1$ in $\Omega \setminus \omega$ and $\chi=0$ in $\tilde{\Omega} \setminus \Omega$.
		Applying now \cref{cor_Holder} for $\tilde{\omega},\, \tilde{B},\, \tilde{\Omega}$ and $\chi u$ we get
		\begin{align*}
		\norm{u}_{H^1(B\setminus \omega)} &\le C \norm{\chi u}_{H^1(\tilde{B}\setminus \tilde{\omega})}
		\le C (\norm{\chi u}_{H^1(\tilde{\omega})} + 
		\norm{\Delta (\chi u) + k^2 \chi u}_{L^2(\tilde{\Omega})})^\alpha \norm{\chi u}_{H^1(\tilde{\Omega})}^{1-\alpha}
		\\& \le C (\norm{u}_{H^1(\omega)} + 
		\norm{\Delta u + k^2 u}_{L^2(\Omega)})^\alpha \norm{u}_{H^1(\Omega)}^{1-\alpha},
		\end{align*}
		where we have used that the commutator $[\Delta,\chi]u$ is supported in $\omega$.		
		A similar proof is valid for the estimates in \cref{cor_Holder_impr} and \cref{shifted3b}.
	\end{proof}
\end{example}

We will give results for two kinds of solutions: a Gaussian bump centred on the top side of the unit square $\Omega$, given in \cref{ex_gauss}, and a variation of the well-known Hamadard solution given in \cref{ex_hadamard}.

\begin{example}\label{ex_gauss} Let the Gaussian bump
	\begin{equation*}
	u = \exp \left( -\frac{(x-0.5)^2}{2\sigma_x} - \frac{(y-1)^2}{2\sigma_y} \right),\quad \sigma_x=0.01,\sigma_y=0.1,
	\end{equation*} 
	be a non-homogeneous solution of (\ref{helmholtz}), i.e. $f = -\Delta u - k^2 u$ and $q=u|_{\omega}$.
\end{example}

\begin{figure}	
	\centering
	\includegraphics[draft=false, width=0.4\textwidth]{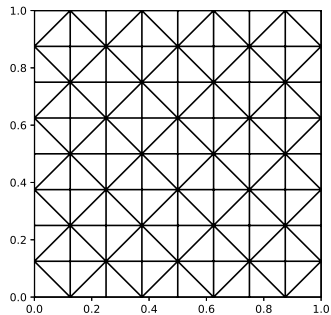}
	\caption{Mesh example.}
	\label{mesh}
\end{figure}

\cref{gauss_conv_10} shows that for \cref{ex_gauss}, when $k=10$, the numerical results strongly agree with the convergence rates expected from \cref{L2error} and \cref{H1error}, and \cref{tnormerror}, i.e. sub-linear convergence for the relative error in the $L^2$- and $H^1$-norms, and quadratic convergence for $\J(u_h,u_h)$. Although in \cref{gauss_conv_50} we do obtain smaller errors and better than expected convergence rates when $k=50$, various numerical experiments indicate that this example's behaviour when increasing the wave number $k$ is rather a particular one. For oscillatory solutions, such as those in \cref{ex_hadamard}, with fixed $n$, or the homogeneous $u = \sin(k x /\sqrt{2}) \cos(k y /\sqrt{2})$, we have noticed that the stability deteriorates when increasing $k$.

In the non-convex setting we let
\begin{equation}\label{non-convex}
\omega = (0.25,0.75) \times (0,0.5),\quad
B = (0.125,0.875) \times (0,0.95),
\end{equation}
and the concentric disks
\begin{equation}\label{disk}
\omega = D((0.5,0.5),0.25),\quad
B = D((0.5,0.5),0.45),
\end{equation}
respectively shown in \cref{domain_nonconv} and \cref{domain_disk}, and we notice from \cref{gauss_nonconv} that the stability strongly deteriorates when one continues the solution outside the convex hull of $\omega$, as the error sizes and rates worsen.

We test the data perturbations by polluting $f$ and $q$ in (\ref{helmholtz}) with uniformly distributed values in $[-h,h]$, respectively $[-h^2,h^2]$, on every node of the mesh. It can be seen in \cref{gauss_perturbation} that the perturbations are visible for an $O(h)$ amplitude, but not for an $O(h^2)$ one.

\begin{figure}
	\begin{subfigure}[t]{0.32\textwidth}
		\includegraphics[draft=false, width=\textwidth]{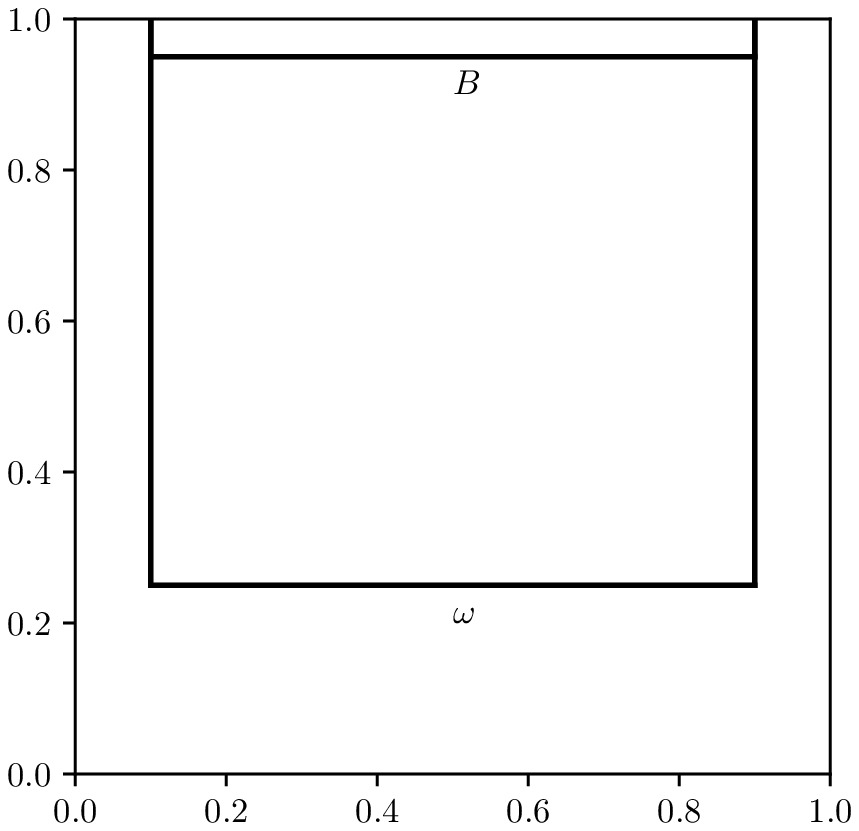}
		\caption{Convex direction (\ref{convex}).}
		\label{domain_conv}
	\end{subfigure}
	\begin{subfigure}[t]{0.32\textwidth}
		\includegraphics[draft=false, width=\textwidth]{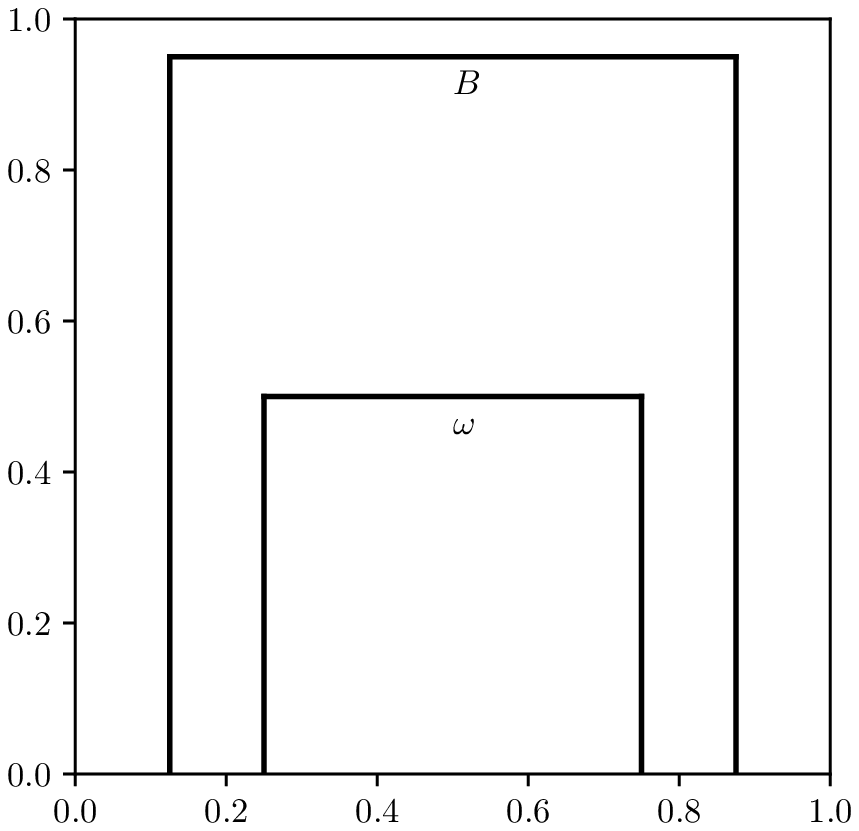}
		\caption{Non-convex direction (\ref{non-convex}).}
		\label{domain_nonconv}
	\end{subfigure}
	\hfill
	\begin{subfigure}[t]{0.32\textwidth}
		\includegraphics[draft=false, width=\textwidth]{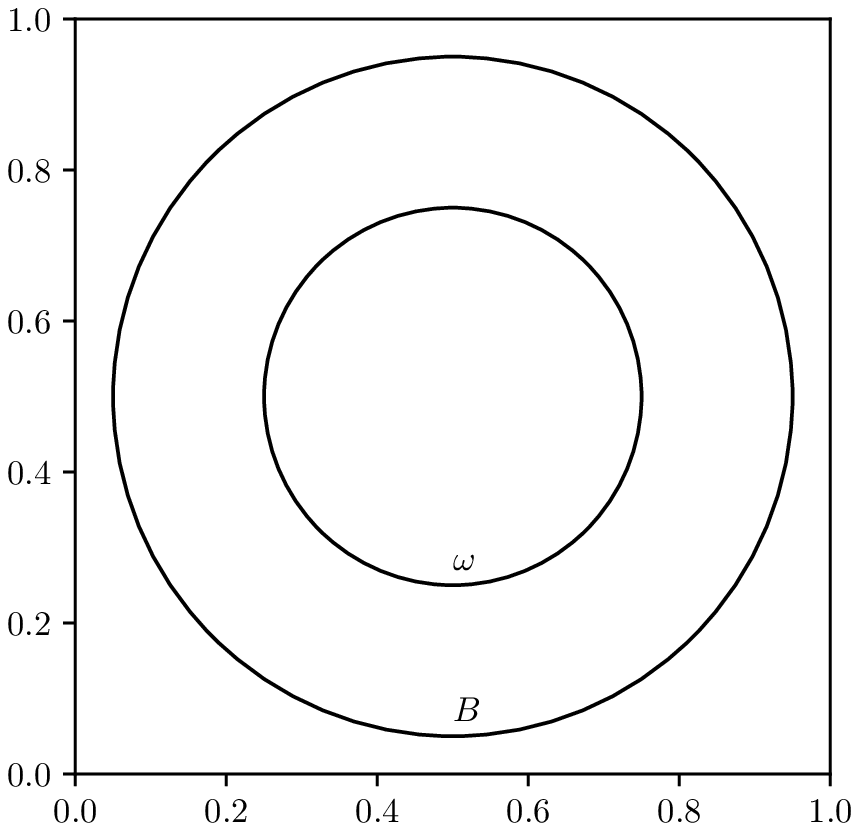}
		\caption{Non-convex direction (\ref{disk}).}
		\label{domain_disk}
	\end{subfigure}
	\caption{Computational domains for \cref{ex_gauss}.}
\end{figure}

\begin{figure}
	\begin{subfigure}[t]{0.48\textwidth}
		\includegraphics[draft=false, width=\textwidth]{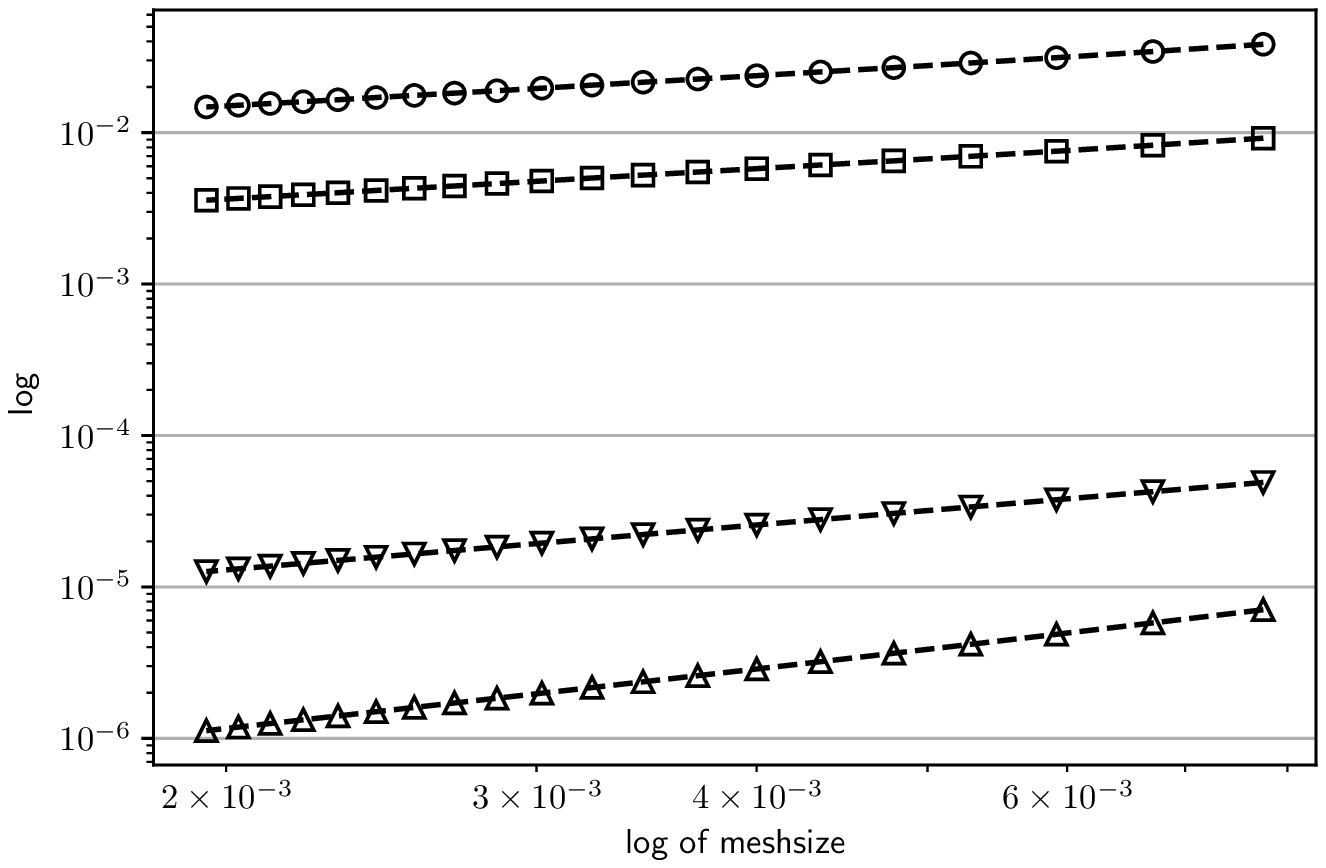}
		\caption{$k=10$. Circles: $H^1$-error, rate $\approx 0.64$; squares: $L^2$-error, rate $\approx 0.66$; down triangles: $h^{-1} \J(u_h,u_h)$, rate $\approx 1$; up triangles: $\norm{z}_W$, rate $\approx 1.3$.}
		\label{gauss_conv_10}
	\end{subfigure}
	\hfill
	\begin{subfigure}[t]{0.48\textwidth}
		\includegraphics[draft=false, width=\textwidth]{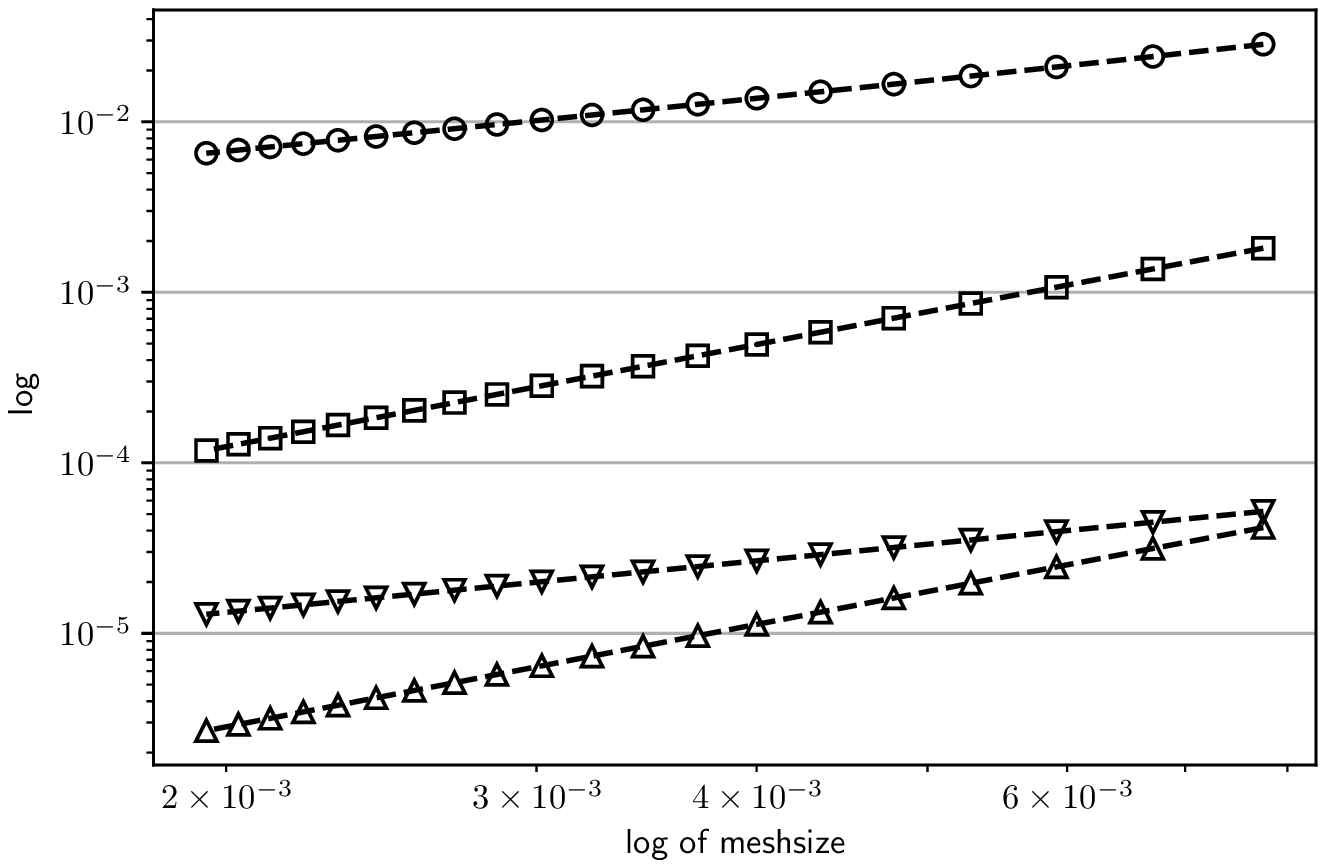}
		\caption{$k=50$. Circles: $H^1$-error, rate $\approx 1.02$; squares: $L^2$-error, rate $\approx 2$; down triangles: $h^{-1} \J(u_h,u_h)$, rate $\approx 1$; up triangles: $\norm{z}_W$, rate $\approx 2$.}
		\label{gauss_conv_50}
	\end{subfigure}
	\caption{Convergence in $B$ for \cref{ex_gauss} in the convex direction (\ref{convex}).}
\end{figure}

\begin{figure}
	\begin{subfigure}[t]{0.48\textwidth}
		\includegraphics[draft=false, width=\textwidth]{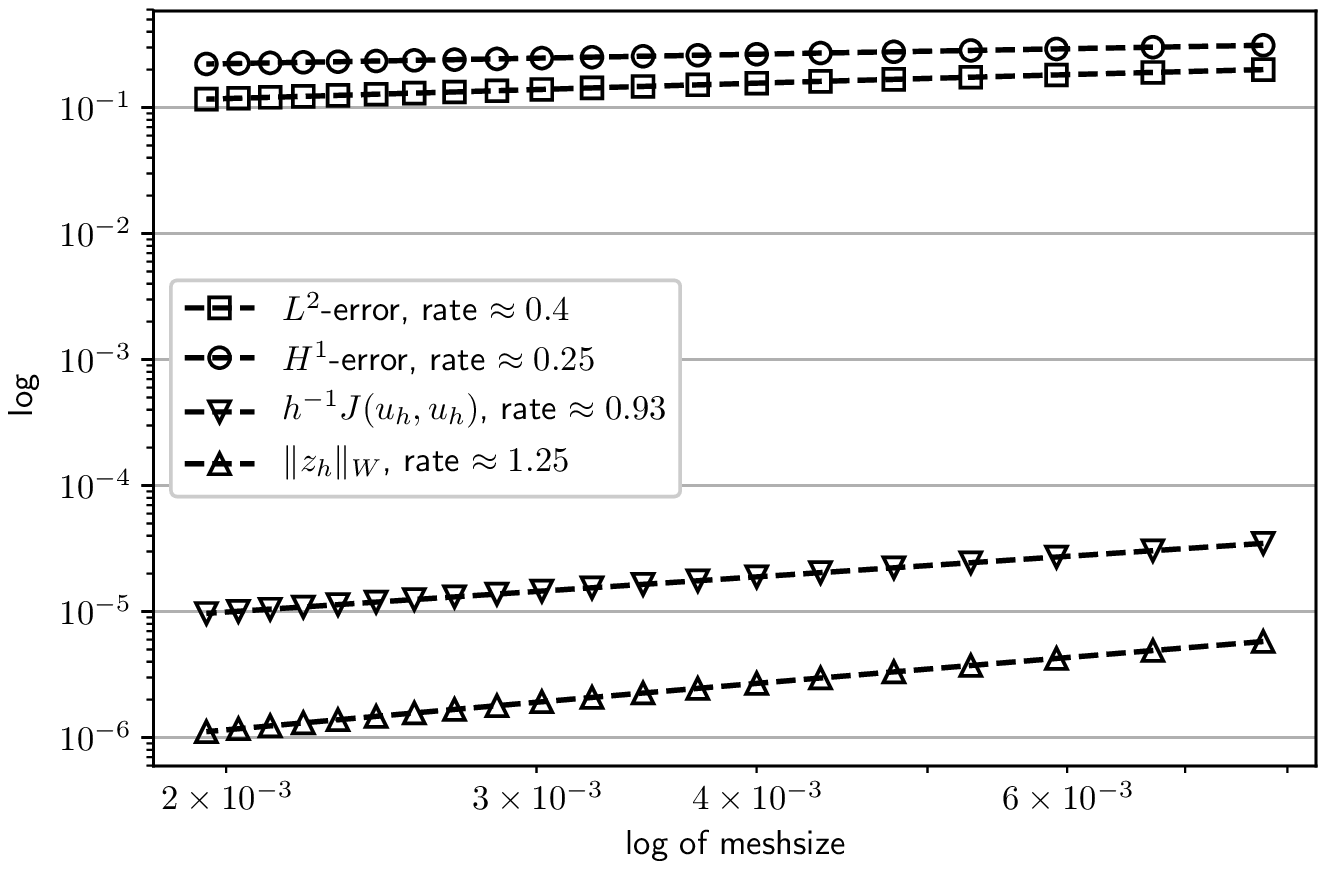}
		\caption{Non-convex direction (\ref{non-convex}).}
	\end{subfigure}
	\hfill
	\begin{subfigure}[t]{0.48\textwidth}
		\includegraphics[draft=false, width=\textwidth]{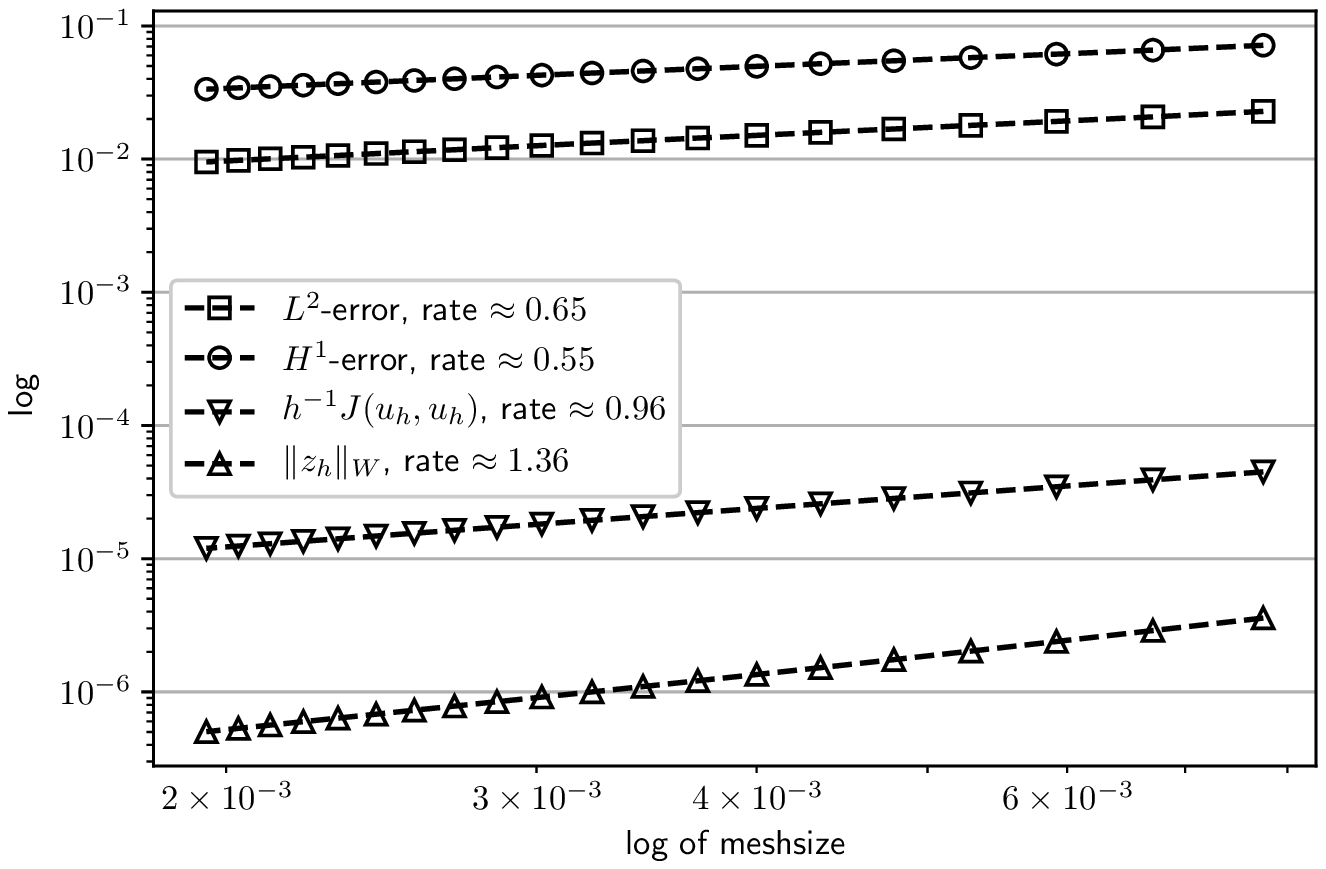}
		\caption{Non-convex direction (\ref{disk}).}
	\end{subfigure}
	\caption{Convergence in $B$ for \cref{ex_gauss}, $k=10$.}
	\label{gauss_nonconv}
\end{figure}

\begin{figure}
	\begin{subfigure}[t]{0.48\textwidth}
		\includegraphics[draft=false, width=\textwidth]{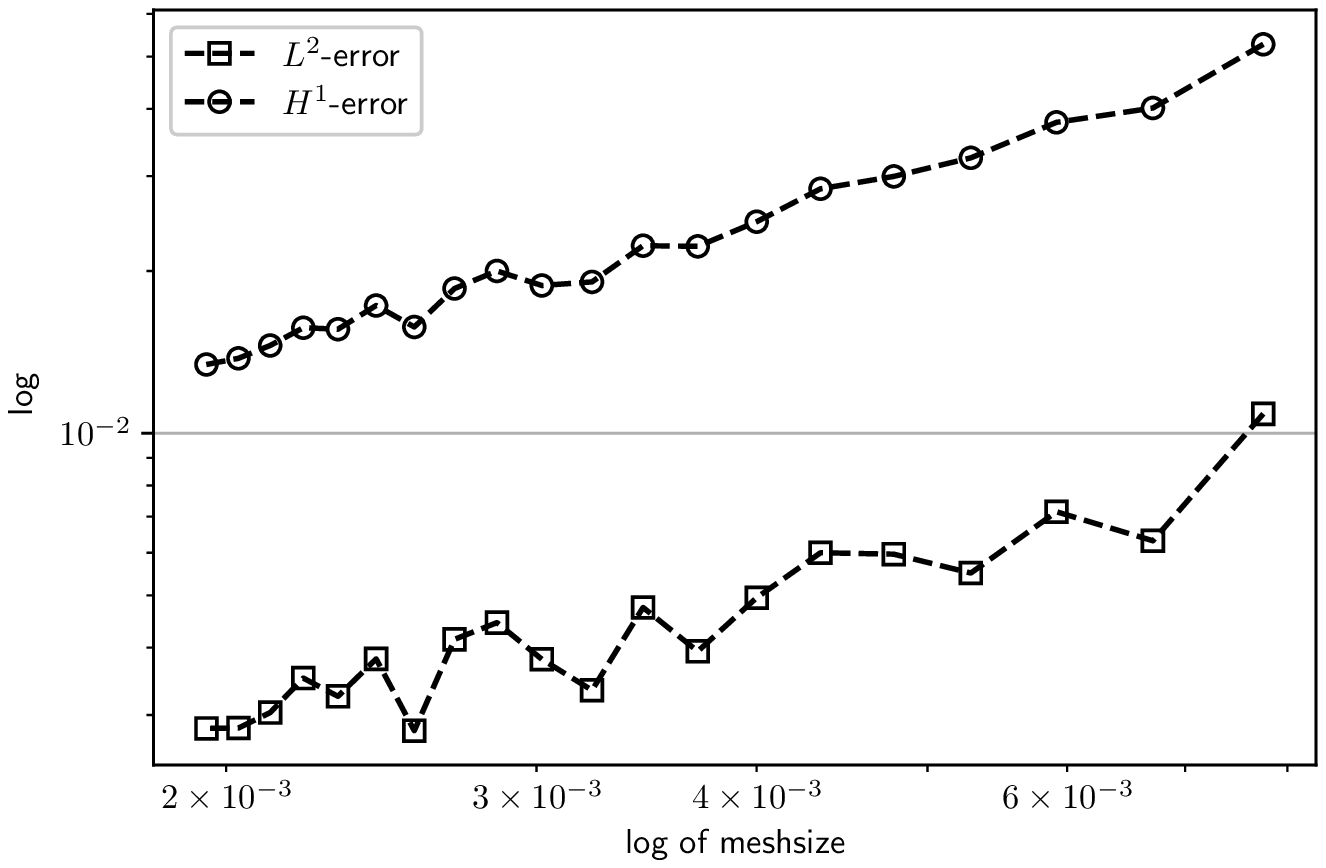}
		\caption{Perturbation $O(h)$.}
	\end{subfigure}
	\hfill
	\begin{subfigure}[t]{0.48\textwidth}
		\includegraphics[draft=false, width=\textwidth]{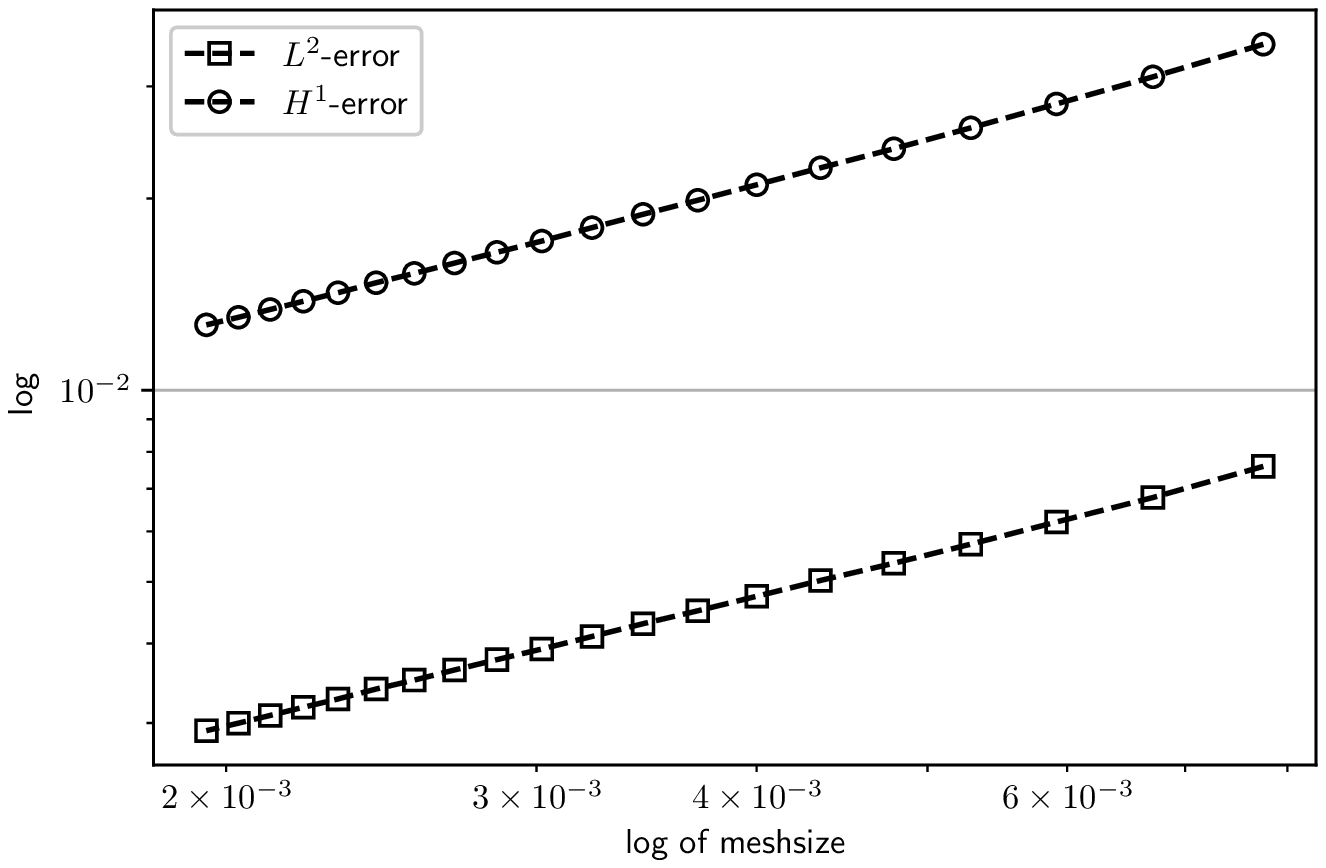}
		\caption{Perturbation $O(h^2)$.}
	\end{subfigure}
	\caption{Convergence in $B$ when perturbing $f$ and $q$ in \cref{ex_gauss} for (\ref{convex}), $k=10$.}
	\label{gauss_perturbation}
\end{figure}

Let us recall that the stability estimates for the unique continuation problem are closely related to those for the notoriously ill-posed Cauchy problem, see e.g. \cite{Alessandrini} or \cite{Isakov}. It is of interest to consider the following variation of a well-known example due to Hadamard, since this example can be used to show that conditional H\"older stability is optimal for the unique continuation problem.

\begin{example}\label{ex_hadamard} Let $n \in \N$ and consider the Cauchy problem 
	\begin{align*}
	\begin{cases}
	\Delta u + k^2 u = 0 \quad &\text{in } \Omega=(0,\pi)\times (0,1), \\
	u(x,0) = 0 \quad &\text{for } x\in [0,\pi], \\
	u_y (x,0) = \sin(nx) \quad &\text{for } x\in [0,\pi],
	\end{cases}
	\end{align*}
	whose solution for $n>k$ is given by $u = \frac{1}{\sqrt{n^2-k^2}} \sin(nx) \sinh(\sqrt{n^2-k^2} y)$, for $n=k$ by $u = \sin(kx) y$, and for $n<k$ by $u = \frac{1}{\sqrt{k^2-n^2}} \sin(nx) \sin(\sqrt{k^2-n^2}y)$. 
\end{example}

It can be seen in \cref{hadamard_conv_fig} that the convergence rates agree with the ones predicted for the convex setting
\begin{equation}\label{hadamard_conv}
\omega = \Omega \setminus [\pi/4,3\pi/4] \times [0,0.25],\quad
B = \Omega \setminus [\pi/4,3\pi/4] \times [0,0.95],
\end{equation}
i.e. sub-linear convergence for the relative error in the $L^2$- and $H^1$-norms, and quadratic convergence for $\J(u_h,u_h)$, although one can notice that the values of the jump stabilizer $\J(u_h,u_h)$ visibly increase compared to \cref{ex_gauss}.

When continuing the solution in the non-convex direction, the stability strongly deteriorates and for coarse meshes the numerical approximation doesn't reach the convergence regime, as it can be seen in \cref{hadamard_nonconv_fig} for the non-convex setting
\begin{equation}\label{hadamard_nonconv}
\omega = (\pi/4,3\pi/4) \times (0,0.5),\quad
B = (\pi/8,7\pi/8) \times (0,0.95).
\end{equation}  

\begin{figure}[t]
	\begin{subfigure}[t]{0.48\textwidth}
		\includegraphics[draft=false, width=\textwidth]{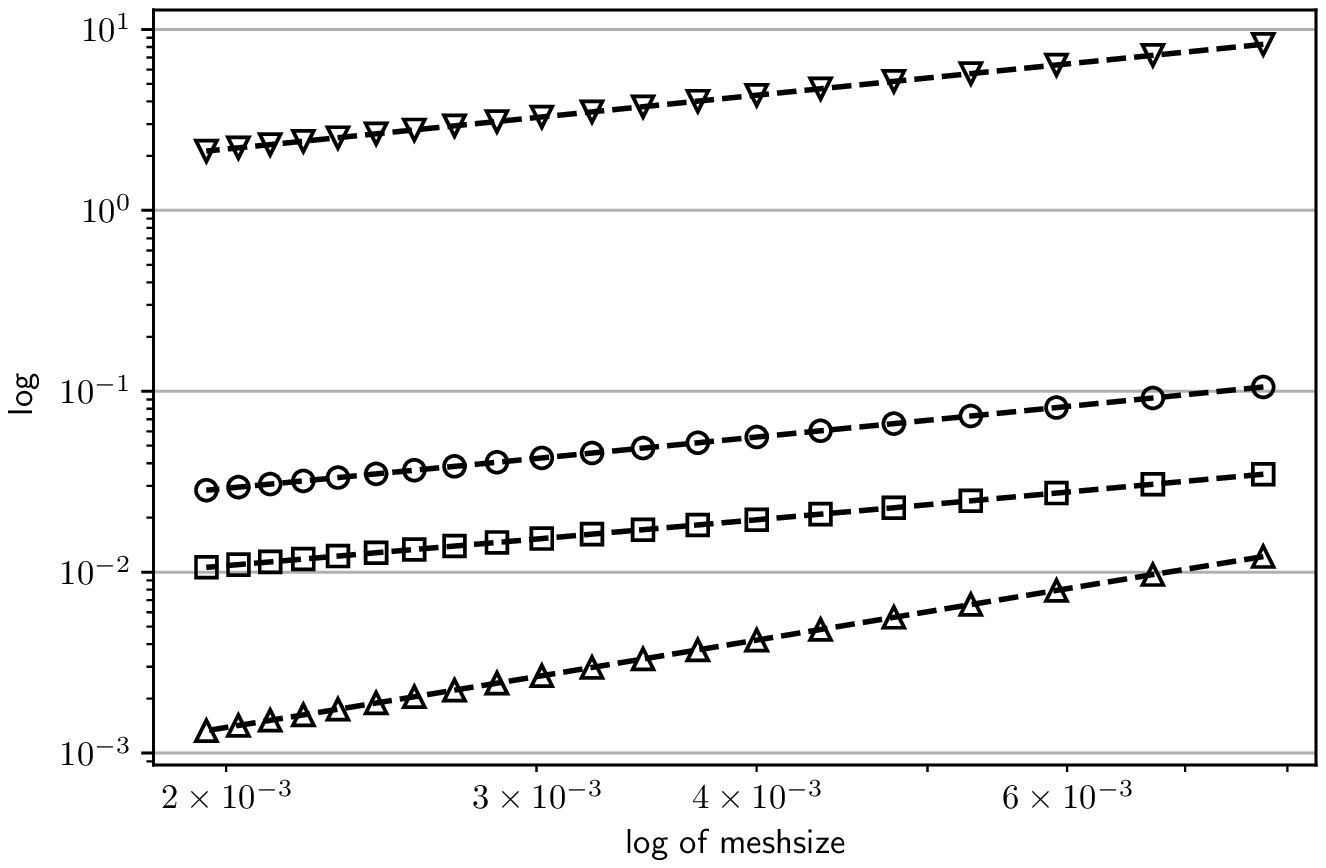}
		\caption{Convex direction (\ref{hadamard_conv}). Circles: $H^1$-error, rate $\approx 0.94$; squares: $L^2$-error, rate $\approx 0.83$; down triangles: $h^{-1} \J(u_h,u_h)$, rate $\approx 1$; up triangles: $\norm{z}_W$, rate $\approx 1.6$.}
		\label{hadamard_conv_fig}
	\end{subfigure}
	\hfill
	\begin{subfigure}[t]{0.48\textwidth}
		\includegraphics[draft=false, width=\textwidth]{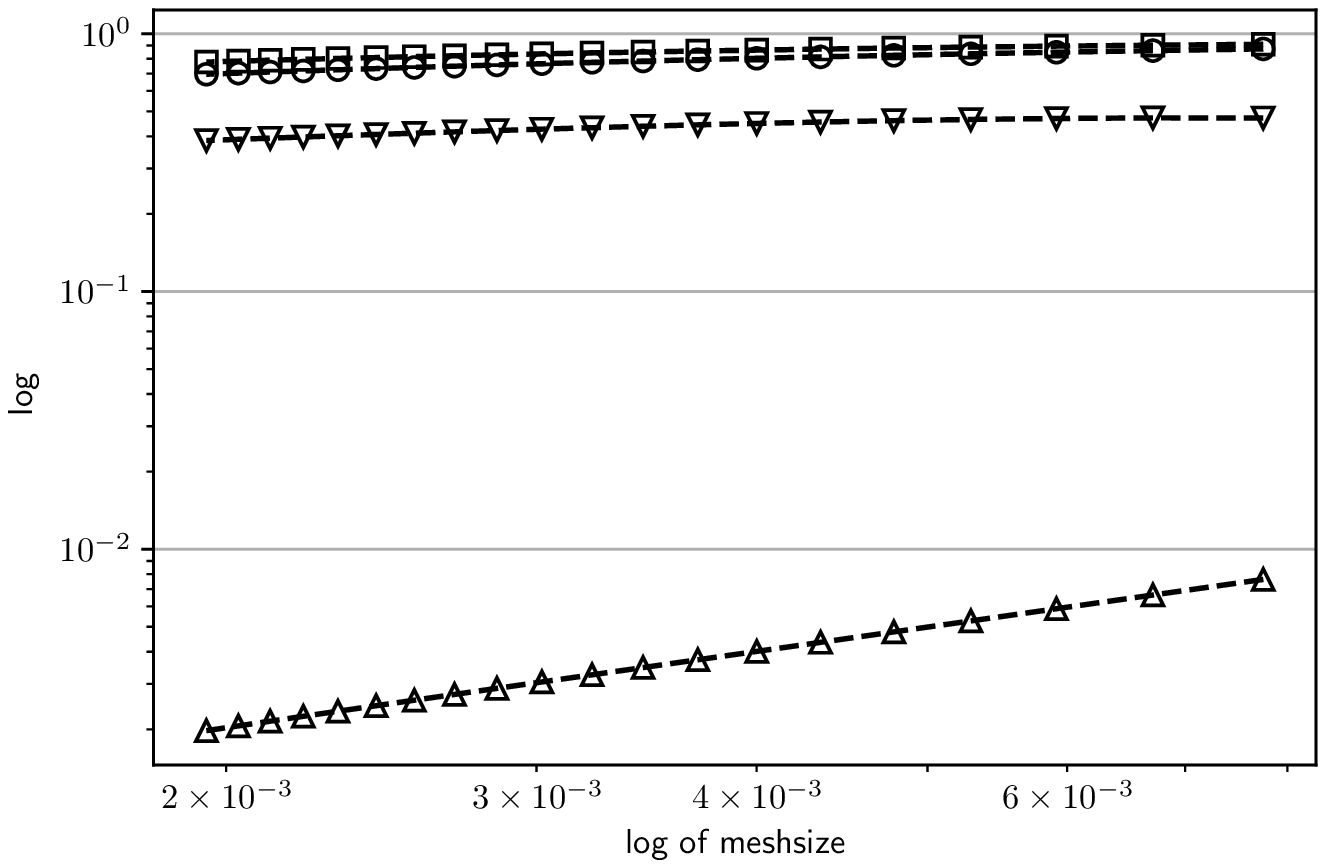}
		\caption{Non-convex direction (\ref{hadamard_nonconv}). Circles: $H^1$-error; squares: $L^2$-error; down triangles: $h^{-1} \J(u_h,u_h)$; up triangles: $\norm{z}_W$.}
		\label{hadamard_nonconv_fig}
	\end{subfigure}
	\caption{Convergence in $B$ for \cref{ex_hadamard}, $k=10$, $n=12$.}
\end{figure}

\appendix
\section{}\label{appendix}

\begin{example}\label{wkb}
	Consider the geometry $\Omega = (0,1)^2$, $\omega = (0,1) \times (0,\epsilon)$ and $B = (0,1) \times (0,1-\epsilon)$,
	and the ansatz $u(x,y) = e^{ikx} a(x,y)$. Let $n\in \N$ and $a(x,y) = a_0(x,y) + k^{-1} a_{-1}(x,y) + \ldots + k^{-n} a_{-n} (x,y)$.
	We have that 
	$$
	\Delta u + k^2 u = e^{ikx} \left( 2 i k \p_x a + \Delta a \right),
	$$
	and we choose $a_j$, $j=0,\ldots,-n$ such that
	\begin{equation}\label{transports}
	\p_x a_0 = 0,
	\quad 2i \p_x a_j + \Delta a_{j+1} = 0,
	\quad -j=1,\ldots,n.
	\end{equation}
	Then
	$$\Delta u + k^2 u = e^{ikx} k^{-n} \Delta a_{-n}$$
	and $\norm{\Delta u + k^2 u}_{L^2(\Omega)} = k^{-n} \norm{\Delta a_{-N}}_{L^2(\Omega)}$. Since $a_j$, $j=0,\ldots,-n$, are independent of $k$ we obtain
	$$\norm{\Delta u + k^2 u}_{L^2(\Omega)} = C k^{-n}.$$
	We can solve (\ref{transports}) such that $a_0(x,y)=a_0(y), \, \supp(a_0) \subset (\epsilon, 1-\epsilon)$ and $\supp(a) \subset [0,1] \times (\epsilon, 1-\epsilon)$.
	Then
	$$
	u|_{\omega} = 0,
	\quad \text{and } \norm{u}_{H^1(B)} = \norm{u}_{H^1(\Omega)} = C k,
	\quad \text{for large } k. 
	$$
	The estimate (\ref{stability_intro}) then becomes
	$$
	k \le C k^{-\alpha n} k^{1-\alpha}, \quad \text{i.e. } k^{\alpha (n+1)} \le C.
	$$
	Choosing large $n$ we see that $C$ depends on $k$, and for any $N \in \N$, $C \le k^N$ cannot hold.
\end{example}

\begin{proof}[Proof of \cref{lem_carleman_eq}]
	Recall the following identities for a function $w$ and vector fields $X$ and $Y$
	\begin{align*}
	\div (w X) 
	= (\grad w, X) + w \div X,
	\quad
	D^2 w(X, Y) = (D_X \grad w, Y),
	\end{align*}
	where $D_X$ is the covariant derivative.
	Recall also that the Hessian is symmetric, i.e. $D^2 w(X, Y) = D^2 w(Y, X)$. We have 
	$$
	e^{\ell} \Delta w
	= \Delta v + b + (q-k^2)v 
	= \Delta v - \sigma v - 2(\nabla v, \nabla \ell) + \sigma v - (\Delta \ell) v + |\nabla \ell|^2 v.
	$$
	Indeed
	\begin{align*}
	\Delta v 
	&= \div( \grad (e^\ell w)) 
	= \div (v \grad \ell + e^\ell \grad w)
	\\&= (\grad v, \grad \ell) + v \Delta \ell + (\grad e^\ell, \grad w) + e^\ell \Delta w
	\\&= 2(\grad v, \grad \ell) +  (\Delta \ell - |\nabla \ell|^2) v + e^\ell \Delta w,
	\end{align*}
	where we have used the identity
	\begin{align*}
	(\grad e^\ell, \grad w) &= (e^\ell \grad \ell, \grad w) 
	= (\grad \ell, \grad (e^\ell w)) - (\grad \ell, w \grad e^\ell)
	\\&= (\grad \ell, \grad v) - v |\grad \ell|^2.
	\end{align*}
	Thus 
	\begin{align}
	\label{squared}
	e^{2\ell} (\Delta w + k^2 w)^2/2
	&= (\Delta v  + b + q v)^2/2
	= (\Delta v  + q v)^2/2 + b^2/2
	+ b\Delta v + bqv,
	\end{align}
	and it remains to study the cross terms $b\Delta v$ and $bqv$. 
	
	Let us begin by studying $\beta \Delta v$ where $\beta = -2(\nabla v, \nabla \ell)$.
	We have
	\begin{align*}
	\beta \Delta v 
	&= \div(\beta \grad v) - (\grad \beta, \grad v)
	\end{align*}
	and
	\begin{align*}
	- (\grad \beta, \grad v) &= 2(\grad (\grad v, \grad \ell), \grad v) 
	= 2(D_{\grad v} \grad v, \grad \ell) + 2(\grad v, D_{\grad v} \grad \ell)
	\\&= 2D^2 v(\grad v, \grad \ell) + 2D^2 \ell(\grad v, \grad v).
	\end{align*}
	Finally 
	\begin{align}
	\label{Hess_expansion}
	2D^2 v(\grad v, \grad \ell) 
	&= 2D^2 v(\grad \ell, \grad v) = 2(D_{\grad \ell} \grad v, \grad v)
	= (\grad \ell, \grad |\grad v|^2)
	\\\nonumber&= \div(|\grad v|^2 \grad \ell) - |\grad v|^2 \Delta \ell.
	\end{align}
	To summarize, for $\beta = -2(\nabla v, \nabla \ell)$
	it holds that 
	\begin{align}
	\label{crossterm1}
	\beta \Delta v 
	&= - \Delta \ell |\grad v|^2 + 2D^2 \ell(\grad v, \grad v)
	+ \div(\beta \nabla v + |\grad v|^2 \grad \ell).
	\end{align}
	
	Consider now $\beta \Delta v$ where $\beta = -\sigma v$.
	We have 
	\begin{align*}
	- (\grad \beta, \grad v)
	= (\grad \sigma, \grad v) v + \sigma |\nabla v|^2,
	\end{align*}
	whence for
	$\beta = -\sigma v$
	it holds that 
	\begin{align}
	\label{crossterm2}
	\beta \Delta v 
	&= \sigma |\grad v|^2 
	+ \div(\beta \nabla v) + (\grad \sigma, \grad v) v.
	\end{align}
	Now (\ref{crossterm1}) and (\ref{crossterm2}) imply 
	\begin{align}
	\label{crossterm12}
	b \Delta v 
	&= a |\grad v|^2 + 2D^2 \ell(\grad v, \grad v)
	+ \div(b \nabla v + c_0) + R_0,
	\end{align}
	where $c_0 = |\grad v|^2 \grad \ell$ and $R_0 = (\grad \sigma, \grad v) v$.
	
	Let us now study the second cross term in (\ref{squared}).
	We have
	\begin{align*}
	-2 (\grad v, \grad \ell) q v  
	= -(\grad v^2, q \grad \ell) 
	=  v^2 \div (q \grad \ell) - \div( v^2 q \grad \ell),
	\end{align*}
	whence, recalling that  $q = k^2 + a + |\nabla \ell|^2$
	and $-a = -\sigma + \Delta \ell$,
	\begin{align}
	\label{crossterm3}
	b q v &= 
	(-\sigma q + \div (q \grad \ell))v^2 + \div c_1
	\\\nonumber&= (-|\nabla \ell|^2 \sigma + \div (|\nabla \ell|^2 \grad \ell))v^2  - k^2 a v^2 + \div c_1 + R_1,
	\end{align}
	where $c_1 = - qv^2 \grad \ell$ and $R_1 = (\div (a \grad \ell) - a\sigma)v^2$.
	The identity (\ref{Hess_expansion}) with $v = \ell$
	implies that 
	$$
	\div (|\nabla \ell|^2 \grad \ell) 
	= 2 D^2 \ell(\nabla \ell, \nabla \ell) + |\nabla \ell|^2 \Delta \ell,
	$$
	whence, recalling that $\sigma = a + \Delta \ell$,
	\begin{align}
	\label{zeroth_coeff}
	-|\nabla \ell|^2 \sigma + \div (|\nabla \ell|^2 \grad \ell)
	= -a |\nabla \ell|^2 + 2 D^2 \ell(\nabla \ell, \nabla \ell).
	\end{align}
	The claim follows by combining (\ref{zeroth_coeff}),
	(\ref{crossterm3}), (\ref{crossterm12}) and (\ref{squared}).
\end{proof}

\textbf{Acknowledgements.} Erik Burman was supported by EPSRC
grants EP/P01576X/1 and EP/P012434/1. Lauri Oksanen was supported by EPSRC grants EP/L026473/1 and EP/P01593X/1.
\bibliographystyle{abbrv}
\bibliography{biblio}
	
\end{document}